\documentclass[reqno]{amsart}

\usepackage{soul}
\usepackage{pgfplots}
\usepackage{amsthm}
\usepackage{latexsym}
\usepackage{dsfont}
\usepackage{bbm}
\usepackage{amssymb}
\usepackage{amsmath}
\usepackage{mathtools}
\usepackage{xcolor}
\usepackage{tikz}
\usepackage{enumerate}

\newcommand{\defeq}{\vcentcolon=}
\newcommand{\eqdef}{=\vcentcolon}
\newcommand{\Law}{\mathcal{L}}

\newcommand{\weaklyto}{%
  \mathrel{\vbox{\offinterlineskip\ialign{%
    \hfil##\hfil\cr
    $\scriptscriptstyle\mathrm{w}$\cr
    $\to$\cr
}}}}
\newcommand{\distto}{%
  \mathrel{\vbox{\offinterlineskip\ialign{%
    \hfil##\hfil\cr
    $\scriptscriptstyle\mathrm{d}$\cr
    \noalign{\kern-.05ex}
    $\to$\cr
}}}}
\newcommand{\fddistto}{%
  \mathrel{\vbox{\offinterlineskip\ialign{%
    \hfil##\hfil\cr
    $\scriptscriptstyle\mathrm{fdd}$\cr
    \noalign{\kern-.05ex}
    $\longrightarrow$\cr
}}}}

\newcommand{\Probto}{%
  \mathrel{\vbox{\offinterlineskip\ialign{%
    \hfil##\hfil\cr
    $\scriptscriptstyle\Prob$\cr
    \noalign{\kern-.05ex}
    $\to$\cr
}}}}
\newcommand{\TVto}{%
  \mathrel{\vbox{\offinterlineskip\ialign{%
    \hfil##\hfil\cr
    $\scriptscriptstyle\mathrm{TV}$\cr
    \noalign{\kern-.05ex}
    $\to$\cr
}}}}

\newcommand{\N}{\mathbb{N}}

\newcommand{\R}{\mathbb{R}}
\newcommand{\B}{\mathcal{B}}

\newcommand{\Gen}{\mathcal{G}}

\newcommand{\I}{\mathcal{I}}
\newcommand{\J}{\mathcal{J}}

\newcommand{\Surv}{\mathcal{S}}

\newcommand{\Prob}{\mathbb{P}}
\DeclareMathOperator{\Var}{\mathrm{Var}}
\DeclareMathOperator{\Cov}{\mathrm{Cov}}

\newcommand{\E}{\mathbb{E}}

\newcommand{\F}{\mathcal{F}}

\renewcommand{\H}{\mathcal{H}}
\newcommand{\cZ}{\mathcal{Z}}
\newcommand{\1}{\mathbbm{1}}

\newcommand{\eqdist}{%
  \mathrel{\vbox{\offinterlineskip\ialign{%
    \hfil##\hfil\cr
    $\scriptscriptstyle\mathrm{law}$\cr
    \noalign{\kern.2ex}
    $=$\cr
}}}}

\newcommand{\dt}{\mathrm{d} \mathit{t}}
\newcommand{\du}{\mathrm{d} \mathit{u}}

\newcommand{\dx}{\mathrm{d} \mathit{x}}

\theoremstyle{plain}
\newtheorem{theorem}{Theorem}[section]
\newtheorem{corollary}[theorem]{Corollary}
\newtheorem{lemma}[theorem]{Lemma}

\theoremstyle{remark}
\newtheorem{remark}[theorem]{Remark}
\theoremstyle{definition}

\numberwithin{equation}{section}

\begin{document}

\title[Gaussian fluctuations of Nerman's martingale]{Gaussian fluctuations
and a law of the iterated logarithm for Nerman's martingale in the supercritical general branching process}

\authors{Alexander~Iksanov, Konrad~Kolesko and Matthias~Meiners}

\address{Alexander~Iksanov, Faculty of Computer Science and Cybernetics,
Taras Shevchenko National University of Kyiv, Ukraine}
	\email{iksan@univ.kiev.ua}

\address{Konrad~Kolesko, Mathematisches Institut, University of Gie\ss en, Germany
and Mathematical Institute, University of
Wroc{\l}aw, Poland}
	\email{konrad.kolesko@math.uni-giessen.de}

\address{Matthias~Meiners, Mathematisches Institut, University of Gie\ss en, Germany}
	\email{matthias.meiners@math.uni-giessen.de}

\keywords{Asymptotic fluctuations, functional central limit theorem, law of the iterated logarithm,
Nerman's martingale, supercritical general branching process}
\subjclass[2010]{60J80,	%Branching processes
			 60F05,	%Central limit and other weak theorems
			 60F17}	%Functional limit theorems; invariance principles

\begin{abstract}
In his, by now, classical work from 1981, Nerman
made extensive use of a crucial martingale $(W_t)_{t \geq 0}$
to prove convergence in probability, in mean and almost surely,
of supercritical general branching processes (also known as Crump-Mode-Jagers branching processes)
counted with a general characteristic.
The martingale terminal value $W$ figures in the limits of his results.

We investigate the rate at which the martingale, now called \emph{Nerman's martingale},
converges to its limit $W$. More precisely, assuming the existence of a Malthusian parameter $\alpha > 0$
and $W_0\in L^2$, we prove a functional central limit theorem for $(W-W_{t+s})_{s\in\R}$, properly normalized, as $t\to\infty$.
The weak limit is a randomly scaled time-changed Brownian motion. Under an additional technical assumption, we prove a law of the iterated logarithm for $W-W_t$.
\end{abstract}

\maketitle

\section{Introduction}		\label{sec:Intro}

The general (Crump-Mode-Jagers) branching process is a classical model
for an evolving population.
The process starts with one initial ancestor at time $0$ which produces
offspring at the times of a point process $\xi$ on $(0,\infty)$.
Every other individual in the process reproduces according to an independent copy of $\xi$
shifted by the individual's time of birth.
The model contains a variety of other models such as Galton-Watson processes,
age-dependent branching processes, Bellman-Harris processes and Sevast'yanov processes.
Counted with a random characteristic the general branching process offers a lot of flexibility in modelling
and allows considerations of, for instance, the number of individuals in the population
in some random phase of life or having some random age-dependent property.
We refer to \cite{Jagers:1975} for a textbook introduction.

The weak and strong laws of large numbers for supercritical general branching processes counted with a random characteristic
were given by Nerman \cite{Nerman:1981}.
Assuming the existence of some $\alpha > 0$ such that $m(\alpha)=1$
where $m$ is the Laplace transform of the intensity measure of the reproduction point process $\xi$,
these laws exhibit exponential growth of the order $e^{\alpha t}$ of the process,
i.e., $\alpha$ is a Malthusian exponent.
Key to the proof of these results is a crucial martingale $(W_t)_{t \geq 0}$,
nowadays called Nerman's martingale.
The martingale limit $W$ appears in the aforementioned weak and strong laws.
Further, the rate of convergence of the martingale to its limit is relevant for
the rate of convergence in the weak law of large numbers for the general branching process \cite{Iksanov+Meiners:2015b}.

Recently, Janson \cite{Janson:2018}
studied the fluctuations of supercritical general branching processes in the case
where $\xi$ is concentrated on a lattice.
A natural first step towards extending
Janson's results to the non-lattice case is to investigate
the asymptotic fluctuations of Nerman's martingale around its limit.
In the paper at hand, we address this problem by proving a functional central limit theorem with a deterministic scaling for $W-W_t$.
This functional limit theorem is complemented by a law of the iterated logarithm.

The results of the present paper are analogous to those for Biggins' martingale
in the branching random walk \cite{Iksanov+Kabluchko:2016}.
Fluctuations of the latter martingale, also at complex parameters,
have received a lot of attention lately \cite{Iksanov+Kolesko+Meiners:2019, Iksanov+Kolesko+Meiners:2021, Roesler+al:2002}.
The endmost paper \cite{Roesler+al:2002} is in the more general context of weighted branching processes.
Predecessors of these results are central limit theorems for the classical martingale in the Galton-Watson process
\cite{Heyde:1970,Heyde:1971,Heyde+Brown:1971}.
There are further rate-of-convergence results for multitype Galton-Watson processes,
we refrain from providing references here and refer to the discussion in \cite{Maillard+Pain:2019} instead.

Closely related to Biggins' martingale in the branching random walk is the derivative martingale,
the fluctuations of which have been addressed in \cite{Buraczewksi+al:2020}.
The counterpart for the derivative martingale in branching Brownian motion is contained in \cite{Maillard+Pain:2019}.
Rate-of-convergence results for more complicated branching processes,
including branching diffusions and superprocesses, can be found in \cite{Ren+Song+Zhao:2019} and the references therein.

\section{Model and assumptions}		\label{sec:model}

We begin by introducing the standard Ulam-Harris notation in the context of the general branching process.
We mainly follow \cite{Jagers:1989}.
Let $\I \defeq \bigcup_{n \in \N_0} \N^n$ be the infinite Ulam-Harris tree
where $\N = \{1,2,\ldots\}$, $\N_0 = \N \cup \{0\}$, and $\N^0 = \{\varnothing\}$
contains only the empty tuple, which we denote by $\varnothing$.
We identify individuals in a population with their descent, which is encoded by elements of $\I$.
For instance, $\varnothing$ is the label of the ancestor, and if $u=(u_1,\ldots,u_m) \in \I$,
then $u$ is the $u_m$th child of the $u_{m-1}$th child of $\ldots$ of the $u_1$th child of the ancestor $\varnothing$.
We abbreviate $u=(u_1,\ldots,u_m)$ by $u_1 \ldots u_m$ and set $|u|$ for the generation of $u$. Here, $|u|=m$.
Similarly, if $v = (v_1,\ldots,v_n)$, we write $uv$ for $(u_1, \ldots, u_m, v_1,\ldots,v_n)$.
Further, if $k \leq m$, we write $u|_k$ for $u_1 \ldots u_k$, the ancestor of $u$ in the $k$th generation.
If $u = v|_m$ for some $0 \leq m \leq |v|$, i.e., when $u$ is an ancestor of $v$, then we write $u \preceq v$,
and say that $v$ stems from $u$. For a subset $F$ of $\I$, we say that $v$ stems from  $F$, and in this case write $F \preceq v$, if $u \preceq v$ for some $u \in F$.
For $F,L \subseteq \I$, we write $F \preceq L$ if every $v \in L$ stems from some $u \in F$.

\subsection{The model}

Let $(\Omega,\F,\Prob)$ be a probability space
on which a family $(\xi_u)_{u \in \I}$ of independent,
identically distributed (i.i.d.)\ point processes on $(0,\infty)$ is defined.
Formally, each $\xi_u$ is an integer-valued mapping
$\xi_u: \Omega \times \B((0,\infty)) \to [0,\infty]$ such that
\begin{itemize}
	\item for fixed $\omega \in \Omega$, $\xi_u(\omega,\cdot): \B((0,\infty)) \to [0,\infty]$ is a measure,
	\item whereas, for  each Borel set $B \in \B((0,\infty))$,
		the map $\xi_u(\cdot,B):\Omega \to [0,\infty]$ is a random variable.
\end{itemize}
Here, $\B((0,\infty))$ is the Borel $\sigma$-algebra of $(0,\infty)$. We write $\xi_u = \sum_{k=1}^{N(u)} \delta_{X_k(u)}$ for $u \in \I$
where $N(u)\defeq \xi_u((0,\infty))$ is a random variable taking values in $\N_0 \cup \{\infty\}$.
For convenience, we abbreviate $\xi_\varnothing$ to $\xi$, $N(\varnothing)$ to $N$, $X_k(\varnothing)$ to $X_k$ etc.
Further, we define $S(u)$, the time of birth of individual $u$, recursively via
\begin{align*}
S(\varnothing) \defeq 0
\qquad	\text{and}	\qquad
S(uk) \defeq S(u) + X_k(u)		\quad	\text{for }u \in \I \text{ and } k \in \N
\end{align*}
with the convention that $X_k(u) \defeq \infty$ if $k > N(u)$.
If $u \in \I$ with $S(u) = \infty$, then individual $u$ is considered never born.
We define $\Gen_n \defeq \{u \in \N^n: S(u) < \infty\}$ to be the $n$th generation individuals, $n \in \N_0$.
Further, $\Gen \defeq \bigcup_{n \in \N_0} \Gen_n$ denotes the set of all individuals that are ever born.
We set
\begin{equation*}
\I_t \defeq \{u \in \I: S(u|_{|u|-1}) \leq t < S(u) < \infty\}
\end{equation*}
to be the coming generation at time $t$, that is, $\I_t$ is the set of particles born after $t$ whose parents were born at or before time $t$.

Later on, different filtrations will be important. We call a subset $L \subseteq \I$ a \emph{line} if $u \not \preceq v$ for all $u,v \in L$ with $u \not = v$.
For any line $L$, we define $\F_L \defeq \sigma(\xi_u: L \not \preceq u)$.
The $\sigma$-algebra $\F_L$ contains all information about the individuals up to and including the line $L$ in the genealogical tree,
but is independent of all information that comes after that, i.e., after crossing the line $L$,
for instance, the relative birth times of individuals descended from an element of $L$. Of particular importance are the $\sigma$-algebras with $L = \N^n$, the $n$th generation,
\begin{equation}
\F_n \defeq	\F_{\N^n} = \sigma(\xi_u: |u| < n),	\qquad	n \in \N_0.
\end{equation}
The family $(\F_n)_{n \in \N_0}$ forms a filtration of $(\Omega,\F)$. We set $\F_\infty \defeq \sigma(\F_n: n \in \N)$.
The second filtration is the counterpart of the first
when the $n$th generation is replaced by the coming generation at time $t$.
To formally introduce it, we first recall the notion of an \emph{optional line}.
An optional line $\J\subseteq \I$ is a random line  with the property that, for every deterministic line $L \subseteq \I$,
it holds that $\{\J \preceq L\} \in \F_L$. For instance, for every $t \geq 0$, the coming generation at time $t$, $\I_t$,
is an optional line.
Indeed, for every line $L \subseteq \F_L$, we have
\begin{equation*}
\{\I_t \preceq L\} =   \bigcap_{v \in L} \{\I_t \preceq v\} = \bigcap_{v \in L} \{S(v) > t\} \in \F_L.
\end{equation*}
We follow Jagers on p.\;190 of \cite{Jagers:1989} and define, for an optional line $\J$,
\begin{equation*}
\F_\J \defeq \{A \in \F_\infty: A \cap \{\J \preceq L\} \in \F_L \text{ for all lines } L \subseteq \I\}.
\end{equation*}
A key result for us is the strong Markov branching property at optional lines, Theorem 4.14 in \cite{Jagers:1989}.
Finally, for $t \geq 0$, we set $\H_t \defeq \F_{\I_t}$.

\subsection{Basic assumptions}

Let $\mu(\cdot) \defeq \E[\xi(\cdot)]$ be the intensity measure of the reproduction point processes.
It is a measure on $\B((0,\infty))$. Throughout the paper, we assume that $\mu$ is not concentrated on any lattice $h\N_0$, $h>0$.
This assumption is for convenience only, all results have lattice counterparts.
We define
\begin{equation}	\label{eq:m}
m(\theta) \defeq	\int \limits_{(0,\infty)} \! e^{-\theta t} \, \mu(\dt),	\quad	\theta \geq 0.
\end{equation}
The function $m$ is the Laplace transform of the intensity measure $\mu$.

Consider the following assumptions:
\begin{enumerate}[{(A}1)]
	\item
		The process is \emph{supercritical}, i.e., $\E[N] = \mu((0,\infty)) > 1$.
	\item
		There exists a \emph{Malthusian parameter} $\alpha > 0$, i.e., an $\alpha > 0$ satisfying
		\begin{equation}	\label{eq:Malthusian alpha}
		m(\alpha) = \int \limits_{(0,\infty)} \! e^{-\alpha t} \, \mu(\dt) = 1.
		\end{equation}
	\item
		The (right) derivative at $\theta = \alpha$ of the Laplace transform $m$ is finite, i.e.,
		\begin{equation}	\label{eq:-m'(alpha)<infty}
		m'(\alpha) \defeq -\int \limits_{(0,\infty)} \! t e^{-\alpha t} \, \mu(\dt) \in (-\infty,0).
		\end{equation}
	\item
		The (random) Laplace transform of $\xi$ at $\theta=\alpha$ has positive and finite variance, i.e.,
		\begin{equation}	\label{eq:E[Z_1^2]<infty}
		0 < \sigma^2 \defeq \E\bigg[\bigg(\sum_{k=1}^{N} e^{-\alpha X_k} \!-\! 1 \bigg)^{\!\!2} \bigg]
		= \E\bigg[\bigg(\int \limits_{(0,\infty)} \! e^{-\alpha t} \, \xi(\dt) \!-\! 1\bigg)^{\!\!2} \bigg] < \infty.
		\end{equation}
	\item
		There exists a nonincreasing Lebesgue integrable function $g:[0,\infty) \to (0,\infty)$ such that
		\begin{equation}	\label{eq:Nerman's Cond 5.1}
		\E\bigg[\sup_{t\geq 0} \frac{\sum_{k=1}^{N} e^{-\alpha X_k} \1_{[0,X_k)}(t)}{g(t)}\bigg]<\infty.
		\end{equation}
\end{enumerate}
Throughout the paper, we shall assume that (A1) through (A3) hold.
Assumption (A1) guarantees that the survival set $\Surv$ defined by
\begin{equation}	\label{eq:survival set}
\Surv = \{\Gen_n \not = \varnothing \text{ for all } n \in \N_0\}
\end{equation}
satisfies $\Prob(\Surv) > 0$.
While assumption (A4) is required for the central limit theorem for Nerman's martingale,
(A5) is additionally used in the proof of the law of the iterated logarithm.

\subsection{Nerman's martingale}\label{sect:Nerman}

Recall that $\I_t$ is the coming generation at time $t$, i.e.,
the collection of labels of individuals born after time $t$
whose parents were born up to (and including) time $t$.
Put
\begin{equation}	\label{eq:Nerman's martingale}
W_t \defeq \sum_{u \in \I_t} e^{-\alpha S(u)},	\quad	t \geq 0.
\end{equation}
The family $(W_t, \H_t)_{t \geq 0}$
is a nonnegative martingale (Proposition 2.4 in \cite{Nerman:1981}), called Nerman's martingale.
It converges almost surely (a.\,s.) as $t \to \infty$ to a finite limit $W \geq 0$ (Corollary 2.5 in \cite{Nerman:1981}).
For later use, we stipulate $W_t \defeq 1$ for $t<0$.
The martingale is a pure jump process almost surely taking values  in the Skorokhod space $D(\R)$
of right-continuous real-valued functions with left limits at every point (c\`{a}dl\`{a}g functions).
With probability one, there are only finitely many jumps on every given compact set.
Indeed, the martingale jumps at $t \geq 0$ only if, for some $n\in\N$, there are $u_1,\ldots,u_n \in \Gen$
and $S(u_j)=t$ for $j=1,\ldots,n$. In this case
\begin{equation*}
\Delta W_t \defeq W_t - W_{t-} = \sum_{j=1}^n e^{-\alpha S(u_j)} \bigg(\sum_{k \geq 1} e^{-\alpha X_k(u_j)} - 1\bigg).
\end{equation*}
Almost surely, there are only finitely many $u$ with $S(u) \leq t$ for any $t \geq 0$. Indeed,
by the many-to-one lemma \cite[Section 1.3]{Shi:2015},
there is a zero-delayed random walk $(\mathsf{S}_n)_{n \in \N_0}$ on $\R$ with increment law $\Prob(\mathsf{S}_n-\mathsf{S}_{n-1} \in \dx) = e^{-\alpha x} \mu(\dx)$,
$\E[\mathsf{S}_1]=-m'(\alpha)>0$ and
\begin{align*}
\E[N_t]
&= \sum_{n \geq 0} \E \bigg[ \sum_{|u|=n} \1_{\{S(u) \leq t\}}\bigg]	
\leq e^{\alpha t} \sum_{n \geq 0} \E \bigg[ \sum_{|u|=n} e^{-\alpha S(u)} \1_{\{S(u) \leq t\}}\bigg]\\
&= e^{\alpha t} \sum_{n \geq 0} \Prob\big(\mathsf S_n\in (0,t]\big)
<\infty.
\end{align*}
The process $(N_t)_{t \geq 0}$ is also a natural example
of a process the limit of which is given by a constant multiple of Nerman's martingale.
Indeed, if (A1) through (A5) hold ((A4) is not needed), Theorem 5.4 of \cite{Nerman:1981}
implies that $e^{-\alpha t} N_t \to \frac{1}{-\alpha m'(\alpha)} W$ a.\,s.\ as $t \to \infty$.

\subsection{The connection with Biggins' martingale}
Nerman's martingale is related to the corresponding Biggins martingale $(Z_n, \F_n)_{n \in \N_0}$, where
\begin{equation}
Z_n = \sum_{|u|=n} e^{-\alpha S(u)},	\quad	n \in \N_0.
\end{equation}
Since the Biggins martingale is nonnegative,
it converges a.\,s.\  to a finite limit $Z \geq 0$ with $\E[Z] \leq 1$.
Further, it holds that $W = Z$ a.\,s.\ by Theorem 3.3 in \cite{Gatzouras:2000}.

What is more, if (A1) through (A4) hold, then
\begin{equation*}
\frac{m(2\alpha)}{m(\alpha)^2} = m(2\alpha) < 1.
\end{equation*}
and in view of Theorem 1 in \cite{Biggins:1992} or Theorem 2.1 in \cite{Liu:2000} it holds $Z_n \to Z$ in $L^2$. In particular, $\E[W]=\E[Z]=1$ and $\Var[W]=\Var[Z] = \E[(Z-1)^2] < \infty$.
Using the fact that martingale increments are uncorrelated, we calculate $\Var[W]$ as follows:
\begin{align}
\sigma_W^2
&\defeq \Var[W] = \E[(Z-1)^2] = \sum_{n\geq 0}\E[(Z_{n+1}-Z_n)^2]	\notag	\\
&= \sum_{n\geq 0}\E\bigg[ \sum_{|u|=n} e^{-2\alpha S(u)}\bigg] \E[(Z_1-1)^2]
= \frac{\sigma^2}{1-m(2\alpha)},	\label{eq:sigma_alpha^2}
\end{align}
where $\sigma^2=\E[(Z_1-1)^2]$.

\subsection{Fluctuations of Nerman's martingale}	\label{subsec:fluctuations of Nerman's martingale}

Recall the notation $\F_\infty\defeq \sigma(\F_n: n \in \N_0)$ and denote by $\Law(X)$ the law (distribution) of a random variable $X$.
If $X,X_t$, $t \geq 0$ are real-valued random variables,
we write
\begin{equation}    \label{eq:convergence weakly in probability}
\Law(X_t | \H_t) \weaklyto \Law(X | \F_{\infty})    \quad   \text{in $\Prob$-probability as $t\to\infty$}
\end{equation}
(in words, `the distribution of $X_t$ given $\H_t$ converges weakly to the distribution of $X$ given $\F_\infty$ in $\Prob$-probability')
if $\E[f(X_t)|\H_t] \Probto \E[f(X)|\F_\infty]$ as $t \to \infty$ for every bounded continuous function $f:\R \to \R$.
Notice that \eqref{eq:convergence weakly in probability} implies $X_t \distto X$ as $t\to\infty$,
where $\distto$ denotes convergence in distribution in $\R$.

Theorem \ref{Thm:CLT for Nerman's martingale} gives the asymptotic fluctuations of $(W_t)_{t \geq 0}$.
\begin{theorem}	\label{Thm:CLT for Nerman's martingale}
Suppose that (A1) through (A4) hold.
Then
\begin{align}	\label{eq:W_t 1-dim convergence}
\Law\big(e^{\alpha t/2} \big(W-W_t\big) \, \big| \, \H_t \big)
\weaklyto	\Law\Big( \Big(\frac{\sigma^2}{-\alpha m'(\alpha)} W\Big)^{\!1/2} \, \cdot X \, \Big| \, \F_\infty \Big)
\end{align}
in probability as $t \to \infty$
where $X$ is standard normal and independent of $\F_\infty$.
\end{theorem}

Convergence in distribution of $e^{\alpha t/2} (W-W_t)$ as $t \to \infty$ can be strengthened
to convergence in distribution of the stochastic process $(e^{\alpha t/2} (W-W_{t+s}))_{s \in \R}$ as $t \to \infty$
in the Skorokhod space $D(\R)$, equipped with the $J_1$-topology (Chapter 12 in \cite{Billingsley:1999}).
We write `$\Rightarrow$' to denote convergence in distribution of random elements in this space.

\begin{theorem}	\label{Thm:FCLT for Nerman's martingale}
Suppose that (A1) through (A4) hold.
Then
\begin{align}	\label{eq:FCLT for Nerman's martingale}
(e^{\alpha t/2} (W-W_{t+s}))_{s \in \R}
\Rightarrow	\Big(\Big(\frac{\sigma^2}{-\alpha m'(\alpha)} W\Big)^{\!1/2} B_{e^{-\alpha s}}\Big)_{\!\!s \in \R}
\quad	\text{as }	t \to \infty
\end{align}
where $(B_s)_{s \geq 0}$ is a standard Brownian motion independent of $W$.
\end{theorem}

Finally, we deal with the almost sure fluctuations of Nerman's martingale, namely,
we formulate a law of the iterated logarithm.
\begin{theorem}	\label{Thm:LIL}
If (A1) through (A5) hold, then, a.\,s.\  on the survival set $\Surv$,
\begin{align}
\limsup_{t\to\infty} \frac{e^{\alpha t/2}}{\sqrt{\log t}}(W-W_t)
&=\Big(\frac{2\sigma^2}{-\alpha m'(\alpha)} W\Big)^{\!1/2},	\label{eq:limsup LIL} \\
\liminf_{t\to\infty}\frac{e^{\alpha t/2}}{\sqrt{\log t}}(W-W_t)
&=-\Big(\frac{2\sigma^2}{-\alpha m'(\alpha)} W\Big)^{\!1/2}. 	\label{eq:liminf LIL}
\end{align}
\end{theorem}

\begin{remark}	\label{Rem:LIL}
It can be checked that $\E[(W-W_t)^2] = \Var[W] \E \big[\sum_{u\in \I_t}e^{-2\alpha S(u)}\big]$.
Hence, by Lemma \ref{Lem:asymptotics of W_t(2alpha)} below, $e^{\alpha t}\E[(W-W_t)^2]$ converges as $t\to\infty$ to a
positive constant. Thus, $\log t$ in \eqref{eq:limsup LIL} and \eqref{eq:liminf LIL}
can be replaced by the asymptotically equivalent function
$\log\big(|\log \E[(W-W_t)^2]|\big)$. This demonstrates that Theorem \ref{Thm:LIL}
is indeed a law of the {\it iterated} logarithm.
\end{remark}

\section{Proofs of the main results}	\label{sec:Proofs}

We start this section with some basic notation and discussions.

\subsection{Preliminaries}	\label{subsec:preliminaries}

\subsubsection*{The shift operators.}
Suppose that $\psi$ is a function of $(\xi_v)_{v \in \I}$,
all offspring point processes. % in the Crump-Mode-Jagers process.
For a given $u \in \I$, we write $[\psi]_u$ for the very same function
but applied to $((\xi_{uv})_{v \in \I})$. In other words, $[\cdot]_u$ is a shift operator
that shifts the ancestor to $u$.
For instance, we have $[Z_1]_u = \sum_{|v|=1} e^{-\alpha X_v(u)}$.
Further, $[Z]_u = \lim_{n \to \infty} [Z_n]_u$ a.\,s.\  and $[W]_u = \lim_{t \to \infty} [W_t]_u$ a.\,s.\
are the limits of the shifted martingales. Here,
\begin{equation*}
[W_t]_u = \sum_{v \in [\I_t]_u} e^{-\alpha (S(uv) - S(u))}, \quad t \geq 0.
\end{equation*}
If $\psi$ is a function of a real variable $t$ and $(\xi_v)_{v \in \I}$,
i.e., $\psi_t = f(t,(\xi_v)_{v \in \I})$ for some function $f$, then we write
$[\psi_{\cdot}]_u \circ t$ for $f(t,(\xi_{uv})_{v \in \I})$.
This is particularly
useful when $t$ is replaced by a function of $S(u)$, for instance,
in this notation, we have $[\psi_{\cdot}]_u \circ (t-S(u)) = f(t-S(u),(\xi_{uv})_{v \in \I})$.

\subsubsection*{Crump-Mode-Jagers processes.}
For a function $\psi$ as above we may now define $\cZ^{\psi}_t$ \emph{the general branching process counted with a characteristic $\psi$}  (or the Crump-Mode-Jagers processes)  by setting
\begin{equation}
\cZ^{\psi}_t\defeq\sum_{u\in\I}[\psi_{\cdot}]_u \circ (t-S(u)).
\end{equation}
The classical Nerman's result states that, under suitable assumptions,
\begin{equation*}
\cZ^{\psi}_t\to\frac W{-m'(\alpha)}\times\int\E[\psi_t]e^{-\alpha t}\dt\quad\text{a.s.}
\end{equation*}

\subsubsection*{Recursive decomposition.}
With this notation, one deduces the following decomposition for $W_{t+r}$, valid for $t,r \geq 0$,
\begin{equation}	\label{eq:recursive representation for W_t+r}
W_{t+r} = \sum_{u \in \I_t} e^{-\alpha S(u)} [W_{r+t-\cdot}]_u \circ S(u).
\end{equation}
Passing to the limit as $r \to \infty$, we infer (after a careful inspection, see Section 14 in \cite{Biggins+Kyprianou:2004}
or Lemma 4.2 in \cite{Alsmeyer+Kuhlbusch:2010})
\begin{align}	\label{eq:FPE for W}
W = \sum_{u \in \I_t} e^{-\alpha S(u)} [W]_u	\quad	\text{a.\,s.}
\end{align}
if (A1) through (A3) hold and $\E[Z_1 \log^+ Z_1] < \infty$.

\subsubsection*{Nerman's martingale as an $L^2$-martingale.}
If (A1) through (A4) hold, then, according to the discussion preceding \eqref{eq:sigma_alpha^2},
$W = Z \in L^2$. Further, $W_t = \E[W | \H_t]$ a.\,s.\ for all $t \geq 0$,
i.e., $(W_t)_{t \geq 0}$ is an $L^2$-bounded martingale and hence convergent in $L^2$ (with limit $W$, of course).
We write $v_t \defeq \Var[W_t] = \E[(W_t-1)^2]$ for the variance of $W_t$, $t \in \R$.
Since $((W_t-1)^2)_{t \geq 0}$ is a right-continuous submartingale,
the function $t \mapsto v_t$ is nondecreasing and right-continuous.
We now identify $v_t$ for some relevant values of $t$.
Trivially, $v_t = 0$ for $t<0$.
Further, since $W_0 = Z_1$, we have $v_0 = \Var[Z_1] = \sigma^2$.
From $W_t \to W$ in $L^2$ and \eqref{eq:sigma_alpha^2} we finally deduce
\begin{equation}	\label{eq:v_infty}
v_\infty \defeq \lim_{t \to \infty} \E[(W_t-1)^2] = \E[(W-1)^2] = \sigma_W^2 = \tfrac{\sigma^2}{1-m(2\alpha)}.
\end{equation}
Moreover, Doob's maximal $L^2$-inequality (with $p=2$)
gives
\begin{equation}	\label{eq:Doob's max L^2 W_t}
\E \bigg[\sup_{s\ge 0} (W_s-1)^2 \bigg] \leq 4 \E[(W-1)^2],
\end{equation}
 i.e., $M_t \defeq \sup_{0 \leq s \leq t} |W_s-1| \in L^2$.

\subsection{Fluctuations of Nerman's martingale: proofs}

We start with an auxiliary result derived from Nerman's law of large numbers for the general branching process.

\begin{lemma}	\label{Lem:asymptotics of W_t(2alpha)}
Suppose that (A1) through (A3) hold.
Then
\begin{align}	\label{eq:asymptotics of W_t(2alpha)}
e^{\alpha t} \sum_{u \in \I_t} e^{-2\alpha S(u)}
&\Probto \frac{1-m(2\alpha)}{-\alpha m'(\alpha)} W	\quad	\text{as } t \to \infty.
\end{align}
More generally, if $f \in D(\R)$ is a nonnegative bounded function, then, as $t \to \infty$,
\begin{align}	\label{eq:asymptotics of W_t(2alpha) generalized}
e^{\alpha t} \sum_{u \in \I_t} \! e^{-2\alpha S(u)} f(t\!-\!S(u))
&\Probto \frac{W}{-m'(\alpha)} \E\bigg[\sum_{k=1}^{N} e^{-2\alpha X_k} \!\! \int_0^{X_k} \!\!\! e^{\alpha x} f(x\!-\!X_k) \, \dx \bigg].
\end{align}
The convergence in \eqref{eq:asymptotics of W_t(2alpha) generalized} (and thus also in \eqref{eq:asymptotics of W_t(2alpha)})
holds in the stronger
\begin{itemize}	\itemsep0pt
	\item	$L^1$ sense if $\E[Z_1 \log^+Z_1]<\infty$,
	\item	almost sure sense if (A5) holds.
\end{itemize}
\end{lemma}
\begin{proof}
Define
\begin{equation*}
\phi(t) \defeq e^{2\alpha t} \sum_{k=1}^{N} e^{-2\alpha X_k} \1_{[0,X_k)}(t) f(t-X_k),	\quad	t  \in \R
\end{equation*}
and notice that, for any $t \geq 0$, with $\|f\|_\infty \defeq \sup_{x \in \R} f(x)$,
\begin{align}	\label{eq:Nerman's Condition (3.2)}
\E\bigg[ \sup_{s \leq t} \phi(s)\bigg]
&\leq e^{2\alpha t} \E\bigg[\sum_{k=1}^{N} e^{-2\alpha X_k} \bigg] \|f\|_\infty = e^{2\alpha t} m(2\alpha) \|f\|_\infty  < \infty,
\end{align}
that is, Condition (3.2) in \cite{Nerman:1981} holds.
Further, for $t \geq 0$,
\begin{align*}
\E[e^{-\alpha t} \phi(t)] = e^{\alpha t} \E\bigg[\sum_{k=1}^{N} e^{-2\alpha X_k} \1_{[0,X_k)}(t) f(t-X_k)\bigg]
= e^{\alpha t} \!\! \int \limits_{(t,\infty)} e^{-2\alpha x} f(t-x) \, \mu(\dx)
\end{align*}
is c\`adl\`ag as a function of $t$ by the dominated convergence theorem.
On the other hand, the inequality
\begin{align*}
\E[e^{-\alpha t} \phi(t)]
&= e^{\alpha t} \E\bigg[\sum_{k=1}^{N} e^{-2\alpha X_k} \1_{\{X_k > t\}} f(t-X_k) \bigg]	\\
&\leq \|f\|_\infty \, \E\bigg[\sum_{k=1}^{N} e^{-\alpha X_k} \1_{\{X_k > t\}} \! \bigg]
= \|f\|_\infty \int \limits_{(t,\infty)} \! e^{-\alpha x} \mu(\dx)
\end{align*}
together with \eqref{eq:-m'(alpha)<infty} shows that the nonnegative c\`adl\`ag function $t\mapsto \E[e^{-\alpha t} \phi(t)]$
is bounded from above by a directly Riemann integrable function on $[0,\infty)$,
hence it is directly Riemann integrable on $[0,\infty)$.
Consequently, the assumptions of Theorem 3.1 in \cite{Nerman:1981}
are satisfied.
The cited theorem gives that, as $t\to\infty$,
\begin{align}
e^{\alpha t} \sum_{u \in \I_t} e^{-2\alpha S(u)}f(t-S(u))
&= e^{-\alpha t} \sum_{u \in \Gen} [\phi]_u(t-S(u))
= e^{-\alpha t} \cZ_t^{\phi}	\notag	\\
&\Probto \frac{W}{-m'(\alpha)} \E\bigg[\sum_{k=1}^{N} e^{-2 \alpha X_k} \int_0^{X_k} e^{\alpha x} f(x-X_k) \, \dx \bigg].	\label{eq:asymptotics of Z_t^phi}
\end{align}
In the special case $f=1$, we find
\begin{align}
e^{\alpha t} \sum_{u \in \I_t} e^{-2\alpha S(u)}
&\Probto \frac{W}{-m'(\alpha)} \E\bigg[\sum_{k=1}^{N} e^{-2 \alpha X_k} \int_0^{X_k} e^{\alpha x} \, \dx \bigg]
= \frac{1-m(2\alpha)}{-\alpha m'(\alpha)} W.	\label{eq:asymptotics of W_t(2alpha)2}
\end{align}
As for the $L^1$-convergence, use Corollary 3.3 in \cite{Nerman:1981}.

Finally, we pass to the a.s.\ convergence.
Since $\phi$ is c\`adl\`ag and
\begin{equation*}
e^{-\alpha t}\phi(t)
= e^{\alpha t} \sum_{k=1}^{N} e^{-2\alpha X_k} \1_{[0,X_k)}(t) f(t-X_k)
\leq \|f\|_\infty \, \sum_{k=1}^{N} e^{-\alpha X_k} \1_{[0,X_k)}(t),
\end{equation*}
condition \eqref{eq:Nerman's Cond 5.1}, which is Condition 5.1 from \cite{Nerman:1981},
entails
\begin{equation*}
\E\bigg[\sup_{t \geq 0} \frac{e^{-\alpha t}\phi(t)}{g(t)}\bigg]<\infty.
\end{equation*}
This is
Condition 5.2 from \cite{Nerman:1981} with the particular $\phi$
and $h=g$. Thus, the assumptions of Theorem 5.4 in
\cite{Nerman:1981} are satisfied.
According to this theorem, the convergence in \eqref{eq:asymptotics of W_t(2alpha)} holds also in the almost sure sense.
\end{proof}

The lemma has the following corollary,
which we use in the proofs of Theorems \ref{Thm:FCLT for Nerman's martingale} and \ref{Thm:LIL},
but not in the proof of Theorem \ref{Thm:CLT for Nerman's martingale}.

\begin{corollary}	\label{Cor:asymptotics of W_t(2alpha)}
Suppose that (A1) through (A4) hold.
Then, for any fixed $\delta > 0$, as $t \to \infty$,
\begin{align}	\label{eq:asymptotics of W_t(2alpha) generalized c_delta}
e^{\alpha t} \sum_{u \in \I_t} \! e^{-2\alpha S(u)} v_{\delta + t - S(u)}
&\to c_\delta W	\quad	\text{in } L^1
\end{align}
where
\begin{equation}	\label{eq:c_delta}
c_\delta \defeq \frac{1}{-m'(\alpha)} \E\bigg[\sum_{k=1}^{N} e^{-2\alpha X_k} \!\! \int_0^{X_k} \!\!\! e^{\alpha x} v_{\delta + x - X_k} \, \dx \bigg].
\end{equation}
The convergence in \eqref{eq:asymptotics of W_t(2alpha) generalized c_delta}
holds in the stronger almost sure sense if (A5) holds.
Further, $c_\delta$ is nondecreasing as a function of $\delta$ with $c_\delta > 0$ for every $\delta > 0$.
The limits of $c_\delta$ as $\delta \to 0$ and $\delta \to \infty$ are given by
\begin{equation}	\label{eq:limits of c_delta}
\lim_{\delta \downarrow 0} c_\delta = 0
\quad	\text{and}	\quad
c_\infty \defeq \lim_{\delta \uparrow \infty} c_\delta = \frac{\sigma^2}{-\alpha m'(\alpha)}.
\end{equation}
\end{corollary}
\begin{proof}
The validity of (A4) implies that the function $t \mapsto v_t$ is bounded by \eqref{eq:v_infty}.
From the discussion preceding \eqref{eq:v_infty}, we infer that $t \mapsto v_t$ is nondecreasing and right-continuous, hence c\`adl\`ag.
Thus, for any fixed $\delta > 0$,
\eqref{eq:asymptotics of W_t(2alpha) generalized c_delta} follows from Lemma \ref{Lem:asymptotics of W_t(2alpha)}
with $f(t) = v_{\delta + t}$, $t \in \R$.
We infer $c_\delta>0$ from $0 < \sigma^2 = v_0 \leq v_t$ for $t \geq 0$ and the representation
\begin{align*}
\int \limits_{0}^{X_k} \! e^{\alpha x} v_{\delta + x - X_k} \, \dx
&=
\1_{\{X_k \leq \delta \}}\int \limits_{0}^{X_k} \! e^{\alpha x} v_{\delta+x-X_k}\,\dx
+ \1_{\{X_k>\delta\}} \!\! \int \limits_{X_k-\delta}^{X_k} \! e^{\alpha x} v_{\delta+x-X_k}\,\dx
> 0.
\end{align*}
Since $t\mapsto v_t$ is a nondecreasing function, so is $\delta \mapsto c_\delta$.
Hence, $\lim_{\delta \downarrow 0} c_\delta$ and $c_\infty
\defeq\lim_{\delta \uparrow \infty} c_\delta$ exist.
Since $v_t = 0$ for $t<0$ and $\lim_{t\to\infty} v_t = \sigma^2/(1-m(2\alpha))$,
\eqref{eq:limits of c_delta} follows with the help of the monotone convergence theorem.
\end{proof}

\begin{proof}[Proof of Theorem \ref{Thm:CLT for Nerman's martingale}]
For $t \geq 0$, we infer from \eqref{eq:FPE for W}
\begin{align*}
e^{\alpha t/2} \big(W-W_t\big)
&= e^{\alpha t/2} \sum_{u \in \I_t} e^{-\alpha S(u)} ([W]_u-1)	\quad	\text{a.\,s.},
\end{align*}
which given $\H_t$ is a weighted sum of independent, centered and square-integrable random variables (Theorem 4.14 in \cite{Jagers:1989}).
We show that the distribution of this sum given $\H_t$ converges in probability
to the distribution of a centered normal random variable.
%To this end, we check the Lindeberg-Feller condition.
For simplicity, we write $\E_t[\cdot]$ to denote the conditional expectation given $\H_t$.
Then, by \eqref{eq:sigma_alpha^2},
\begin{align}
\E_t\bigg[\bigg(e^{\alpha t/2}\sum_{u \in \I_t} e^{-\alpha S(u)} ([W]_u-1)\bigg)^{\!2}\bigg]
&=e^{\alpha t} \E_t \bigg[\sum_{u \in \I_t} e^{-2\alpha S(u)} ([W]_u-1)^2\bigg]	\notag	\\
&= \frac{\sigma^2}{1-m(2\alpha)} e^{\alpha t} \sum_{u \in \I_t} e^{-2\alpha S(u)}.		\label{eq:conditional var}
\end{align}
Observe that $\E[Z_1^2]<\infty$ entails $\E[Z_1\log^+ Z_1]<\infty$ whence, by Lemma \ref{Lem:asymptotics of W_t(2alpha)},
\eqref{eq:asymptotics of W_t(2alpha)} holds in $L^1$ and thereupon
\begin{align}	\label{eq:con var in L^1}
\E_t\bigg[\bigg(e^{\alpha t/2}\sum_{u \in \I_t} e^{-\alpha S(u)} ([W]_u-1)\bigg)^{\!2}\bigg]
&\to \frac{\sigma^2}{-\alpha m'(\alpha)} W	\quad	\text{in } L^1
\end{align}
as $t \to \infty$.
Further, for $x \geq 0$, we define
\begin{equation*}
\sigma^2_W(x) \defeq \E[(W-1)^2 \1_{\{|W-1| > x\}}]
\end{equation*}
and notice that $\lim_{x\to\infty}\sigma^2_W(x)=0$
by \eqref{eq:v_infty} and the dominated convergence theorem.
Consequently, for all $\varepsilon>0$,
\begin{multline*}
\sum_{u \in \I_t} \E_t\big[\big(e^{\alpha t/2} e^{-\alpha S(u)} ([W]_u-1)\big)^2
\1_{\{|e^{\alpha t/2} e^{-\alpha S(u)} ([W]_u-1)| > \varepsilon\}} \big]		\\
\quad	= e^{\alpha t} \sum_{u \in \I_t} e^{-2 \alpha S(u)} \sigma^2_W(\varepsilon e^{-\alpha t/2} e^{\alpha S(u)})
\leq \sigma^2_W(\varepsilon e^{\alpha t/2}) e^{\alpha t} \sum_{u \in \I_t} e^{-2 \alpha S(u)}
\Probto 0
\end{multline*}
as $t \to \infty$ by \eqref{eq:asymptotics of W_t(2alpha)}.
Thus, (2.5) through (2.7) in \cite{Helland:1982} hold, and we conclude \eqref{eq:W_t 1-dim convergence}.
\end{proof}

For the proof of the functional central limit theorem, Theorem \ref{Thm:FCLT for Nerman's martingale},
we need some preparatory lemmas.

\begin{lemma}	\label{Lem:sigma^2}
Suppose that (A1) through (A4) hold.
Then the family $((W-W_t)^2)_{t \in \R}$ is uniformly integrable.
In other words, the function
$\sigma^2_t(x) \defeq \E[|W-W_t|^2 \1_{\{|W-W_t| > x\}}]$, $t \in \R$, $x \geq 0$
is a bounded  on $\R \times [0,\infty)$ with
\begin{equation}	\label{eq:sigma^2(r,x)->0 unif}
\sup_{t \in \R} \sigma^2_t(x) \to 0	\quad	\text{as } x \to \infty.
\end{equation}
\end{lemma}
\begin{proof}
From the discussion preceding \eqref{eq:v_infty}, we know that $W_t \to W$ a.\,s.\  and in $L^2$,
and that $W_t = \E[W|\H_t]$ a.\,s.
Thus, the family $(W_t^2)_{t \geq 0}$ is uniformly integrable,
hence so is the family $((W\!-\!W_t)^2)_{t \geq 0}$,
which implies \eqref{eq:sigma^2(r,x)->0 unif}.
\end{proof}

\begin{lemma}	\label{Lem:Cov_t}
Suppose that (A1) through (A4) hold and let $0 \leq r < s < \infty$. Then, as $t \to \infty$
and with $\Cov_t[\cdot,\cdot]$ and $\Var_t[\cdot]$ denoting conditional covariance and variance given $\H_t$, respectively,
\begin{align}
\Cov_t[e^{\alpha t/2} (W\!-\!W_{t+r}), e^{\alpha t/2} (W\!-\!W_{t+s})]
&= \Var_t[e^{\alpha t/2} (W\!-\!W_{t+s})]	\notag	\\
&\to \frac{e^{-\alpha s} \sigma^2}{-\alpha m'(\alpha)} W	\quad	\text{in } L^1.	\label{eq:Cov_t}
\end{align}
\end{lemma}
Notice that the expression on the right-hand side of \eqref{eq:Cov_t}
is exactly the (conditional) covariance of the limiting process in \eqref{eq:FCLT for Nerman's martingale}.
\begin{proof}
For any $t \geq 0$, we have
\begin{align*}
\Cov_t&[e^{\alpha t/2} (W\!-\!W_{t+r}), e^{\alpha t/2} (W\!-\!W_{t+s})]	\\
&= \Cov_t[e^{\alpha t/2} (W\!-\!W_{t+s}+W_{t+s}\!-\!W_{t+r}), e^{\alpha t/2} (W\!-\!W_{t+s})]	\\
&= \Var_t[e^{\alpha t/2} (W\!-\!W_{t+s})] + \Cov_t[e^{\alpha t/2} (W_{t+s}\!-\!W_{t+r}), e^{\alpha t/2} (W\!-\!W_{t+s})]\\
&=\Var_t[e^{\alpha t/2} (W\!-\!W_{t+s})].
\end{align*}
% Hence, in order to prove the first equality
% in \eqref{eq:Cov_t}, if suffices to show that the last conditional covariance vanishes a.\,s.
% Indeed,
% \begin{align*}
% &\Cov_t[W_{t+s}\!-\!W_{t+r},W\!-\!W_{t+s}]	\\
% &~= \sum_{u \in \I_t} \! e^{-\alpha S(u)} \! \sum_{v \in \I_t} \! e^{-\alpha S(v)}
% \! \Cov_t[[W_{t+s-\cdot} \!-\! W_{t+r-\cdot}]_u \circ S(u), [W \!-\! W_{t+s-\cdot}]_u \circ S(u)]	\\
% &~=	\sum_{u \in \I_t} \! e^{-\alpha S(u)} \! \sum_{v \in \I_t} \! e^{-\alpha S(v)}
% \! \Cov[W_{t+s-\cdot} \!-\! W_{t+r-\cdot}, W \!-\! W_{t+s-\cdot}] \circ S(u) = 0	\quad	\text{a.\,s.}
% \end{align*}
since increments of square-integrable martingales are (conditionally) uncorrelated.
It thus remains to investigate
$\Var_t[e^{\alpha t/2} (W\!-\!W_{t+s})]$ as $t \to \infty$.
Here, arguing as in the proof of Theorem \ref{Thm:CLT for Nerman's martingale} and using Lemma \ref{Lem:asymptotics of W_t(2alpha)},
we infer, with $\sigma^2_t = \E[(W-W_t)^2]$ for $t \in \R$,
\begin{align*}
\Var_t[e^{\alpha t/2} (W\!-\!W_{t+s})]
&= e^{\alpha t} \sum_{u \in \I_t} e^{-2\alpha S(u)} \sigma^2_{t+s-S(u)}	\\
&\to \frac{W}{-m'(\alpha)} \E\bigg[\sum_{k=1}^{N} e^{-2\alpha X_k} \!\! \int_0^{X_k} \!\!\! e^{\alpha x} \sigma^2_{s+x-X_k} \dx \bigg]
\eqdef d_s W
\end{align*}
as $t \to \infty$ in $L^1$, where $d_s \geq 0$ is a constant.
We now calculate the constant $d_s$,
but avoid evaluating it directly.
Instead, notice that
\begin{align*}
\Var_t\big[e^{\alpha t/2} (W-W_{t+s})\big]
&= \E_t\big[\big(e^{\alpha t/2} (W-W_{t+s})\big)^2\big]	\\
&= \E_t\big[\E_{t+s}\big[\big(e^{\alpha t/2} (W-W_{t+s})\big)^2\big]\big].
\end{align*}
By \eqref{eq:con var in L^1}, $\E_{t+s}[(e^{\alpha t/2} (W-W_{t+s}))^2]$ converges in $L^1$
as $t \to \infty$, in particular, the family $(\E_{t+s}[(e^{\alpha t/2} (W-W_{t+s}))^2])_{t \geq 0}$
is uniformly integrable. Hence so is $(\E_t[(e^{\alpha t/2} (W-W_{t+s}))^2])_{t \geq 0}$.
Consequently,
\begin{align*}
\E\big[\E_t\big[\big(e^{\alpha t/2} (W-W_{t+s})\big)^2\big]\big] \to d_s	\quad	\text{as } t \to \infty.
\end{align*}
On the other hand,
\begin{align*}
\E\big[\E_t\big[\big(e^{\alpha t/2} (W\!-\!W_{t+s})\big)^2\big]\big]
&= e^{-\alpha s} \E\big[\E_t\big[\big(e^{\alpha (t+s)/2} (W\!-\!W_{t+s})\big)^2\big]\big]
\to \frac{e^{-\alpha s} \sigma^2}{-\alpha m'(\alpha)}
\end{align*}
by \eqref{eq:con var in L^1}. Hence, $d_s =  \frac{e^{-\alpha s} \sigma^2}{-\alpha m'(\alpha)}$.
\end{proof}

\begin{proof}[Proof of Theorem \ref{Thm:FCLT for Nerman's martingale}]
We first prove weak convergence of the finite-dimensional distributions on $[0,\infty)$.
To this end, fix $n \in \N$ and $0 \leq s_1 < s_2 \ldots < s_n$.
Abbreviate $t+s_k$ by $t_k$, $k=1,\ldots,n$ and define $t_{n+1} \defeq \infty$ and $W_{t_{n+1}} \defeq W$.
We use the Cram\'er-Wold device which reduces the problem to studying the convergence in law of the following linear combinations:
\begin{equation*}
\sum_{k=1}^n \gamma_k e^{\alpha t/2} (W-W_{t_k})
\end{equation*}
for fixed $\gamma_1,\ldots,\gamma_n \in \R$.
Recall that $\Var_t$ and $\Cov_t$ denote the conditional variance and covariance given $\H_t$, respectively.
Lemma \ref{Lem:Cov_t} yields
\begin{align}
\Var_t&\bigg[\sum_{k=1}^n \gamma_k e^{\alpha t/2} (W\!-\!W_{t_k})\bigg]	\notag	\\
&= \sum_{k=1}^n \gamma_k^2 e^{\alpha t} \Var_t[W\!-\!W_{t_k}] + 2 \sum_{1 \leq j < k \leq n} \gamma_j \gamma_k \Cov_t[W\!-\!W_{t_j},W\!-\!W_{t_k}]
\notag	\\
&\to \bigg(\sum_{k=1}^n \gamma_k^2 e^{-\alpha s_k}
+ 2 \sum_{1 \leq j < k \leq n} \gamma_j \gamma_k e^{-\alpha s_k} \bigg) \frac{\sigma^2}{-\alpha m'(\alpha)} W
\quad	\text{in } L^1
\end{align}
as $t\to\infty$. Note that the expression in the last line is exactly the conditional variance of the linear combinations corresponding to the limiting process in \eqref{eq:FCLT for Nerman's martingale} since
\begin{align*}
\Var&\bigg[\sum_{k=1}^n \gamma_k \Big(\frac{\sigma^2}{-\alpha m'(\alpha)} W\Big)^{\!1/2} B_{e^{-\alpha s_k}} \, \bigg| \, W  \bigg]
= \Var\bigg[\sum_{k=1}^n \gamma_k B_{e^{-\alpha s_k}} \bigg] \frac{\sigma^2}{-\alpha m'(\alpha)} W	\\
&= \bigg(\sum_{k=1}^n \gamma_k^2 e^{-\alpha s_k}
+ 2 \sum_{1 \leq j < k \leq n} \gamma_j \gamma_k e^{-\alpha s_k} \bigg) \frac{\sigma^2}{-\alpha m'(\alpha)} W	\quad	\text{a.\,s.}
\end{align*}
Next, we check the Lindeberg-Feller condition. Using the decomposition
\begin{equation*}
\sum_{k=1}^n \gamma_k e^{\alpha t/2} (W-W_{t_k})
= \sum_{k=1}^n \gamma_k e^{\alpha t/2}  \sum_{u \in \I_t} e^{-\alpha S(u)} \big([W]_u-[W_{t_k-\cdot}]_u \circ (S(u))\big),
\end{equation*}
we argue as in the proof of Theorem \ref{Thm:CLT for Nerman's martingale}.
Recalling the notation $\sigma^2_r(x)=\E[|W-W_r|^2 \1_{\{|W-W_r| > x\}}]$, we infer, for all $\varepsilon > 0$,
\begin{align*}
\sum_{u \in \I_t} & \E_t\big[\big(e^{\alpha(t/2-S(u))} ([W]_u\!-\![W_{t_k-\cdot}]_u \circ (S(u)))\big)^2	\\
&	\hphantom{\E_t\big[\big(} \cdot \1_{\{|e^{\alpha(t/2-S(u))} ([W]_u-[W_{t_k-\cdot}]_u \circ (S(u)))| > \varepsilon\}} \big]	\\
&= e^{\alpha t} \sum_{u \in \I_t} e^{- 2 \alpha S(u)} \sigma^2_{t_k-S(u)}(e^{-\alpha t/2} e^{\alpha S(u)} \varepsilon)
\Probto 0
\end{align*}
by \eqref{eq:asymptotics of W_t(2alpha)} and \eqref{eq:sigma^2(r,x)->0 unif}.

This implies
that the finite-dimensional distributions of $(e^{\alpha t/2} (W-W_{t+s}))_{s \geq 0}$ given $\H_t$
converge weakly to those of $((\sigma^2/(-\alpha m'(\alpha)) W)^{\!1/2} B_{e^{-\alpha s}})_{s \geq 0}$ given $\F_\infty$.
From this, we conclude that
\begin{align}\label{eq:FDDCLT for Nerman's martingale}
(e^{\alpha t/2} (W-W_{t+s}))_{s \geq 0}
\fddistto	\Big(\Big(\frac{\sigma^2}{-\alpha m'(\alpha)} W\Big)^{\!1/2} B_{e^{-\alpha s}}\Big)_{\!\!s \geq 0}
\quad	\text{as }	t \to \infty
\end{align}
where $\fddistto$ denotes weak convergence of finite-dimensional distributions.

The next step is to check that the distributions of the family
\begin{equation}	\label{eq:tightness of e^alpha t/2 (W_t-W_t+s) in D[0,infty)}
(e^{\alpha t/2} (W_t-W_{t+s}))_{s \geq 0}, \ t \geq 0 \text{ are tight in } D([0,\infty)).
\end{equation}
We use Aldous's tightness criterion, see e.g.\ Theorem 16.10 on p.\;178 of \cite{Billingsley:1999}.
To this end, we first check condition (16.22) in the cited source, i.e.,
\begin{equation}	\label{eq:Billingsley (16.22)}
\lim_{x \to \infty} \limsup_{t \to \infty} \Prob\Big(\sup_{0 \leq s \leq b} |e^{\alpha t/2}(W_t-W_{t+s})| \geq x\Big) = 0
\quad	\text{for all } b > 0.
\end{equation}
For any fixed $b,t,x > 0$, by Chebyshev's inequality and Doob's maximal $L^p$-inequality (with $p=2$),
\begin{align}	\label{eq:(16.22) Chebyshev and Doob's max ineq}
\Prob\Big(\sup_{0 \leq s \leq b} |e^{\alpha t/2}(W_t-W_{t+s})| \geq x\Big)
&\leq \frac{4 e^{\alpha t}}{x^2} \E[|W_{t+b}-W_t|^2]
\quad	\text{for all } b > 0.
\end{align}
Here, an application of \eqref{eq:asymptotics of W_t(2alpha) generalized c_delta} yields
\begin{align*}
e^{\alpha t} \E[|W_{t+b}-W_t|^2]
&= e^{\alpha t} \E\bigg[\bigg(\sum_{u \in \I_t} e^{-\alpha S(u)}([W_{t+b-\cdot}]_u \circ S(u) - 1) \bigg)^2 \bigg]	\\
&= e^{\alpha t} \E\bigg[\sum_{u \in \I_t} e^{-2 \alpha S(u)} v_{t+b-S(u)} \bigg] \to c_b	\quad	\text{as } t \to \infty.
\end{align*}
Using this after taking the $\limsup$ as $t \to \infty$ in \eqref{eq:(16.22) Chebyshev and Doob's max ineq},
and then letting $x \to \infty$ gives \eqref{eq:Billingsley (16.22)}.

We now turn to the second condition of Aldous's criterion, namely, for all $\varepsilon, b > 0$
\begin{equation}	\label{eq:Condition 1}
\lim_{\delta \to 0} \limsup_{t \to \infty} \sup_{\tau} \Prob(e^{\alpha t/2}|(W_t-W_{t+\tau+\delta}) - (W_t-W_{t+\tau})| \geq \varepsilon) = 0
\end{equation}
where $\sup_\tau$ is the supremum over all discrete stopping times $0 \leq \tau \leq b$ with respect to the filtration $(\H_{t+s})_{s \geq 0}$.
Here, `discrete' means that $\tau$ takes only finitely many values.
To prove \eqref{eq:Condition 1}, fix $\varepsilon, b > 0$
and let $\tau \leq b$ be a discrete stopping time. Then $\tau_t \defeq t+\tau$ is a stopping time with respect to $(\H_s)_{s \geq 0}$.
We claim that $\I_{\tau_t}$ is an optional line. Clearly, it is a random line. Further,
if $t \leq t_1 \leq \ldots \leq t_m \leq t+b$ are the values $\tau$ takes, then, for any deterministic line $L \subseteq \I$,
\begin{align*}
\{\I_{\tau_t} \preceq L\} = \bigcup_{k=1}^m (\{\tau_t = t_k\} \cap \{\I_{t_k} \preceq L\}) \in \F_L
\end{align*}
since $\tau_t$ is a stopping time with respect to $(\H_s)_{s \geq 0}$, thus $\{\tau_t = t_k\} \in \H_{t_k} = \F_{\I_{t_k}}$
and, consequently, $\{\tau_t = t_k\} \cap \{\I_{t_k} \preceq L\} \in \F_L$ by the definition of $\F_{\I_{t_k}}$.
We may now use the strong Markov branching property (Theorem 4.14 in \cite{Jagers:1989}) to conclude that
\begin{align*}
\Prob&(e^{\alpha t/2}|(W_t-W_{t+\tau+\delta}) - (W_t-W_{t+\tau})| \geq \varepsilon)
= \Prob(e^{\alpha t/2}|(W_{t+\tau+\delta}) - W_{t+\tau}| \geq \varepsilon)	\\
&\leq \frac{e^{\alpha t}}{\varepsilon^2} \E[|(W_{t+\tau+\delta}) - W_{t+\tau}|^2]	\\
&= \frac{e^{\alpha t}}{\varepsilon^2} \E\bigg[\E\bigg[\bigg(\sum_{u \in \I_{\tau_t}} e^{-\alpha S(u)} ([W_{\tau_t+\delta-\cdot}]_u\circ S(u) - 1)\bigg)^2 \, \bigg| \, \F_{\I_{\tau_t}} \bigg] \bigg]	\\
&=  \frac{e^{\alpha t}}{\varepsilon^2} \E\bigg[\sum_{u \in \I_{t+\tau}} e^{-2 \alpha S(u)} v_{t+\tau+\delta - S(u)}\bigg].
\end{align*}
Here, if $0 \leq s_1 \leq \ldots \leq s_m \leq b$ denote the values of $\tau$, then
\begin{align*}
e^{\alpha t} \sum_{u \in \I_{t+\tau}} & e^{-2 \alpha S(u)} v_{t+\tau+\delta - S(u)}	\\
&= \sum_{k=1}^m e^{\alpha s_k} \1_{\{\tau=s_k\}} e^{-\alpha(t+s_k)} \sum_{u \in \I_{t+s_k}} e^{-2 \alpha S(u)} v_{t+\tau+\delta - S(u)}	\\
&\to \sum_{k=1}^m e^{\alpha s_k} \1_{\{\tau=s_k\}} c_\delta W
= e^{\alpha \tau} c_\delta W	\quad	\text{in } L^1	\text{ as } t \to \infty
\end{align*}
by Corollary \ref{Cor:asymptotics of W_t(2alpha)}.
Hence,
\begin{align*}
\limsup_{t \to \infty} \sup_{\tau} \Prob&(e^{\alpha t/2}|(W_t-W_{t+\tau+\delta}) - (W_t-W_{t+\tau})| \geq \varepsilon)
\leq \frac{1}{\varepsilon^2} e^{\alpha b} c_\delta.
\end{align*}
Thus, \eqref{eq:Condition 1} follows from \eqref{eq:limits of c_delta}.
Combining \eqref{eq:Billingsley (16.22)} and \eqref{eq:Condition 1} with Theorem 16.10 on p.\;178 of \cite{Billingsley:1999}
yields \eqref{eq:tightness of e^alpha t/2 (W_t-W_t+s) in D[0,infty)}.
Since the increments of the processes $(e^{\alpha t/2} (W_t-W_{t+s}))_{s \geq 0}$
and $(e^{\alpha t/2} (W-W_{t+s}))_{s \geq 0}$ are the same,
Theorem 16.5 in \cite{Billingsley:1999} in combination with  the fact that $e^{\alpha t/2}(W-W_t)$ converges in distribution implies that also the distributions of the family $(e^{\alpha t/2} (W-W_{t+s}))_{s \geq 0}$, $t \geq 0$ are tight in $D([0,\infty))$.
Together with the convergence of the finite-dimensional distributions, we obtain
\begin{align}	\label{eq:FCLT for Nerman's martingale on D[0,infty)}
(e^{\alpha t/2} (W\!-\!W_{t+s}))_{s \geq 0}
\Rightarrow	\Big(\Big(\frac{\sigma^2}{-\alpha m'(\alpha)} W\Big)^{\!1/2} \! B_{e^{-\alpha s}}\Big)_{\!\!s \geq 0}
\text{ as }	t \to \infty \text{ on } D([0,\infty)).
\end{align}
For any $r> 0$ the shift operator $\theta_r:f(\cdot)\mapsto f(\cdot +r)$ is an isometry between $D([0,b])$ and $D([-r,b-r])$. Hence, $\theta_r$ is a continuous mapping from $D([0,\infty))$ to $D([-r,\infty))$.
In particular, for any $r>0$, by the scaling invariance of Brownian motion,
we infer
\begin{align*}	%\label{eq:FCLT for Nerman's martingale on D[-r,infty)}
(e^{\alpha t/2} (W\!-\!W_{t+s}))_{s \geq -r}&=
\theta_r\Big((e^{\alpha t/2} (W\!-\!W_{t-r+s}))\Big)_{s \geq -r}\\
&\Rightarrow	\theta_r\Big(e^{\alpha r/2}\Big(\frac{\sigma^2}{-\alpha m'(\alpha)} W\Big)^{\!1/2} \! B_{e^{-\alpha s}}\Big)_{\!\!s \geq -r}\\
&\eqdist \Big(\Big(\frac{\sigma^2}{-\alpha m'(\alpha)} W\Big)^{\!1/2} \! B_{e^{-\alpha s}}\Big)_{\!\!s \geq 0},
\end{align*}
as $t \to \infty$ on $ D([-r,\infty))$. Here, $\eqdist$ denotes equality of distributions. Since this holds for every $r>0$, we arrive at \eqref{eq:FCLT for Nerman's martingale}.
\end{proof}

\subsection{Law of the iterated logarithm: proofs}

The key tool in our proof of the law of the iterated logarithm for Nerman's martingale
is Proposition 7.2 on p.~436 in \cite{Asmussen+Hering:1983}.
\begin{lemma}	\label{Lem:Asmussen+Hering}
Let $(\mathcal{R}_n)_{n\in\N_0}$ be an increasing sequence of
$\sigma$-fields and $(T_n)_{n\in\N_0}$ be a sequence of random
variables such that
\begin{equation}	\label{eq:Berry-Esseen-type condition}
\sum_{n \geq 0} \sup_{y\in\R} \left|\Prob(T_n\leq y \, | \, \mathcal{R}_n) - \Phi(y)\right|<\infty \quad\text{{\rm a.\,s.}}
\end{equation}
where $\Phi(y)\defeq \frac 1 {\sqrt{2\pi}} \int_{-\infty}^y e^{-x^2/2} \, \dx$, $y\in\R$. Then
\begin{equation*}
\limsup_{n\to\infty} \frac {T_n}{\sqrt{\log n}} \leq \sqrt{2} \quad\text{{\rm a.\,s.}}
\end{equation*}
If there is a $k\in\N$ such that $T_n$ is
$\mathcal{R}_{n+k}$-measurable for each $n\in\N_0$, then
\begin{equation*}
\limsup_{n\to\infty} \frac {T_n}{\sqrt{\log n}} = \sqrt{2} \quad\text{{\rm a.\,s.}}
\end{equation*}
\end{lemma}

The next result is Lemma A.2 in \cite{Iksanov+Kabluchko:2016},
an infinite version of the Berry-Esseen inequality for independent,
centered random variables. It is likely that this fact is also given in other sources.

\begin{lemma}	\label{Lem:Berry-Esseen}
Let $Y_1, Y_2,\ldots$ be independent random variables
with $\E[Y_i] = 0$, $\sigma^2_{Y_i} \defeq \Var[Y_i] < \infty$ and $\rho_{Y_i} \defeq \E[|Y_i|^3]$, $i\in\N$.
If $\sum_{i\geq 1} \sigma^2_{Y_i} < \infty$, then, for an absolute constant $C$,
\begin{equation}	\label{eq:Berry-Esseen}
\sup_{y\in\R} \left|\Prob\left(\frac{\sum_{i\geq 1}
Y_i}{(\sum_{i\geq 1} \sigma^2_{Y_i})^{1/2}}\leq y\right) -
\Phi(y) \right| \leq C \frac{\sum_{i\geq 1} \rho_{Y_i}}
{\big(\sum_{i\geq 1} \sigma^2_{Y_i}\big)^{3/2}}.
\end{equation}
\end{lemma}

Recall that the limits in \eqref{eq:limsup LIL} and \eqref{eq:liminf LIL} are
considered on the survival set $\mathcal{S}$. We only give a
complete proof for the upper limit. Investigating $W_t-W$
rather than $W-W_t$ gives the result for the lower limit
at no extra cost. Although the scheme of the proof is similar to
that of Theorem 3.4 on p.~130 in \cite{Asmussen+Hering:1983} in which a Markov branching process was investigated, technical details differ at places. Without loss of generality we
assume in what follows that $\Prob(\Surv)=1$ (otherwise we
have to use Lemma \ref{Lem:Asmussen+Hering} with the probability measure
$\Prob$ replaced with $\Prob(\cdot|\Surv)$ and write ``a.\,s.\
on the survival set $\Surv$'' rather than ``a.\,s.''
throughout). This assumption ensures that $W$ is positive
a.\,s.\ rather than with positive probability.

For $t,r>0$, we use the following representations
derived from \eqref{eq:recursive representation for W_t+r}
and \eqref{eq:FPE for W}:
\begin{align*}
W_{t+r} - W_t &= \sum_{u \in \I_t} e^{-\alpha S(u)}([W_{r+t-\cdot}]_u \circ S(u) \, - 1)	%\label{eq:W-W_t}
\\
\text{and}	\quad	
W-W_t &= \sum_{u \in \I_t}e^{-\alpha S(u)} ([W]_u-1).	%\label{eq:W-W_t}
\end{align*}
Recall that the $S(u)$, $u\in\I_t$ are $\H_t$-measurable,
whereas the $[W_{r-x}]_u$, $u \in \I_t$ and the $[W]_u$, $u \in \I_t$ are independent of $\H_t$, see Theorem 4.14 in \cite{Jagers:1989}.
Since we do not assume $\E [|W_t|^3]<\infty$, we
start by investigating the sums as above with truncated summands.
For $t \geq 0$ and $r \in (0,\infty]$, let
\begin{equation*}
W_{t,r}(u) \defeq e^{-\alpha S(u)}([W_{r+t-\cdot}]_u\circ S(u)-1)  \1_{\{e^{\alpha t/2} e^{-\alpha S(u)}|[W_{r+t-\cdot}]_u \circ S(u) -1|\leq 1\}}
\end{equation*}
and
\begin{equation}	\label{repr2}
V_{t,r} = \sum_{u\in\I_t} \big(W_{t,r}(u)-\E_t[W_{t,r}(u)]\big).
\end{equation}

\begin{lemma}	\label{Lem:asymptotic var V_t,r}
For $r \in (0,\infty]$, with $c_r$ as defined in \eqref{eq:c_delta}, we have
\begin{equation}	\label{eq:asymptotic var V_t,r}
\lim_{t\to\infty} e^{\alpha t} \Var_t [V_{t,r}]= c_r W
\quad	\text{{\rm a.\,s.}}
\end{equation}
\end{lemma}
\begin{proof}
Conditionally on $\H_t$, the random variables $W_{t,r}(u)$, $u\in\I_t$
are independent (but not identically distributed).
By definition of $V_{t,r}$, we have
\begin{equation*}
\Var_t[V_{t,r}]
= \sum_{u\in\I_t} \E_t[W_{t,r}(u)^2] - \sum_{u\in\I_t} \big(\E_t[W_{t,r}(u)]\big)^2
\eqdef G_{t,r}' -G_{t,r}''.
\end{equation*}
It is sufficient to show that
\begin{align}
\lim_{t \to \infty} e^{\alpha t} G_{t,r}'		= c_r W \quad	&\text{a.\,s.}
\label{eq:tech limit 1}	\\
\text{and}	\quad
\lim_{t \to \infty} e^{\alpha t} G_{t,r}''		= 0	 \quad	&\text{a.\,s.}					\label{eq:tech limit 2}
\end{align}
To this end, for $t \in \R$, let $F_t$ and $G_t$
denote the distribution functions of $|W_t-1|$ and $\sup_{0 \leq s \leq t} |W_s-1|$, respectively.
For instance, $F_t(x) = \Prob(|W_t-1|\leq x)$ for $x \in \R$.	\smallskip

\noindent \emph{Proof of \eqref{eq:tech limit 1}.}
We have
\begin{align*}
G_{t,r}'
&= \sum_{u\in\I_t} \E_t \big[W_{t,r}(u)^2\big]
= \sum_{u\in\I_t} e^{-2\alpha S(u)} \int_{[0,\, e^{-\alpha t/2}e^{\alpha S(u)}]} \!\!\! x^2 \, {\rm d}F_{r+t-S(u)}(x).
\end{align*}
For $u\in\I_t$, $S(u) > t$ and therefore, for every $c > 0$,
$e^{-\alpha t/2}e^{\alpha S(u)} \geq e^{\alpha t/2} \geq c$ for all sufficiently large $t$.
Consequently
\begin{align*}
\sum_{u\in\mathcal{I}_t}e^{-2\alpha S(u)}
v_{r+t-S(u)}(c)
\leq G_{t,r}'
\leq \sum_{u\in\I_t} e^{-2\alpha S(u)} v_{r+t-S(u)}
\end{align*}
where $v_s(c) \defeq \int_{[0,\,c]} x^2 \, {\rm d}F_{s}(x)$. Corollary \ref{Cor:asymptotics of W_t(2alpha)} yields
\begin{align}
e^{\alpha t} G_{t,r}' &\leq e^{\alpha t} \sum_{u \in \I_t} \! e^{-2\alpha S(u)} v_{r+t-S(u)}		\notag	\\
&\to \frac{W}{-m'(\alpha)} \E\bigg[\sum_{k=1}^{N} e^{-2\alpha X_k} \!\! \int_0^{X_k} \!\!\! e^{\alpha x} v_{r+x-X_k} \, \dx \bigg]
= c_r W	\label{eq:asymptotics of W_t(2alpha) generalized recalled}
\end{align}
a.\,s.\ and in $L^1$ as $t \to \infty$
where the definition of $c_r$ should be recalled from \eqref{eq:c_delta}.
Analogously,
for any $c > 0$,
\begin{align}
e^{\alpha t} G_{t,r}'
&\geq e^{\alpha t} \sum_{u \in \I_t} \! e^{-2\alpha S(u)} v_{r+t-S(u)}(c)	\notag	\\	
&\to \frac{W}{-m'(\alpha)} \E\bigg[\sum_{k=1}^{N} e^{-2\alpha X_k} \!\! \int_0^{X_k} \!\!\! e^{\alpha x} v_{r+x-X_k}(c) \, \dx \bigg]
\label{eq:asymptotics of W_t(2alpha) generalized recalled 2}
\end{align}
a.\,s.\ and in $L^1$. Since $v_s(c) \uparrow v_s$ as $c \uparrow \infty$, \eqref{eq:tech limit 1}
follows from the monotone convergence theorem.	\smallskip

\noindent \emph{Proof of~\eqref{eq:tech limit 2}.}
Since $\E[[W_s]_u-1]=0$ for all $s \in \R$,
\begin{align*}
G_{t,r}''
&= \sum_{u\in\I_t} \! e^{-2\alpha S(u)}\big(\E \big(([W_{r+t-\cdot}]_u \circ S(u) -1) \1_{\{e^{\alpha t/2}e^{-\alpha S(u)}|[W_{r+t-\cdot}]_u \circ S(u)-1|\leq 1\}} \big)\big)^2\\
&= \sum_{u\in\I_t} \! e^{-2\alpha S(u)}\big(\E \big[([W_{r+t-\cdot}]_u \circ S(u) -1) \1_{\{e^{\alpha t/2}e^{-\alpha S(u)}|[W_{r+t-\cdot}]_u \circ S(u)-1| >  1\}} \big]\big)^2.
\end{align*}
In view of $[W_s]_u-1\leq |[W_s]_u-1|$ and $S(u) > t$ for $u \in \I_t$, we have
\begin{align*}
G_{t,r}'' &\leq
\sum_{u\in\I_t} \bigg(e^{-2\alpha S(u)}\bigg(\int_{(e^{-\alpha t/2} e^{\alpha S(u)}, \infty)} x \, {\rm d} F_{r+t-S(u)}(x)\bigg)^{\!\! 2} \, \bigg)\\
&\leq \sum_{u\in\I_t}e^{-2\alpha S(u)} \bigg(\int_{(c, \infty)} x \, {\rm d} F_{r+t-S(u)}(x)\bigg)^2	\\
&\eqdef \sum_{u\in\I_t}e^{-2\alpha S(u)} \bar v_{r+t-S(u)}(c)
\end{align*}
for all sufficiently large $t$ where $\bar v_x(c)\defeq (\E[|W_x-1|\1_{\{|W_x-1| \geq c\}}])^2$ for $x\geq 0$.
Since $((W_t-1)^2)_{t \geq 0}$ is a submartingale,
\begin{align*}
0 \leq \bar v_t(c)
&\leq % (\E[|W_t-1| \1_{\{|W_t-1| \geq c\}}])^2 \leq
\E[(W_t-1)^2] \leq \E[(W-1)^2] = \sigma_W^2 < \infty,
\end{align*}
i.e., $t\mapsto \bar v_t(c)$ is a nonnegative and bounded function.
Doob's maximal inequality enables us to apply the dominated convergence theorem to show that
the function $t\mapsto \bar v_t(c)$ is c\`adl\`ag.
Consequently, we may apply \eqref{eq:asymptotics of W_t(2alpha) generalized} to conclude that
\begin{align*}
e^{\alpha t} G_{t,r}'' &\leq \sum_{u\in\I_t}e^{-2\alpha S(u)} \bar v_{r+t-S(u)}(c)	\\
&\to \frac{W}{-m'(\alpha)} \E\bigg[\sum_{k=1}^{N} e^{-2\alpha X_k} \!\! \int_0^{X_k} \!\!\! e^{\alpha x} \bar v_{r+x-X_k}(c) \, \dx \bigg]
\quad	\text{a.\,s.\ and in } L^1.
\end{align*}
Further,
\begin{equation*}
\|\bar v_{\cdot}(c)\|_{\infty} \defeq \sup_{t \in \R} v_t(c) = \sup_{t \in \R} \E[|W_t-1| \1_{\{|W_t-1| \geq c\}}] \to 0
\quad	\text{as }	c \to \infty.
\end{equation*}
Therefore, as $c \to \infty$,
\begin{align*}
\E\bigg[\sum_{k=1}^{N} e^{-2\alpha X_k} \!\! \int_0^{X_k} \!\!\! e^{\alpha x} \bar v_{r+x-X_k}(c) \, \dx \bigg]
&= \int \limits_{[0,\infty)} \! e^{-2\alpha t} \! \int_0^t e^{\alpha x} \bar v_{r+x-t}(c) \, \dx \, \mu(\dt)
\\ &\leq \frac{\|\bar v_{\cdot}(c)\|_{\infty}}{\alpha} \int \limits_{[0,\infty)} \! e^{-2\alpha t} (e^{\alpha t}-1) \, \mu(\dt)
\to 0.
\end{align*}
This proves \eqref{eq:tech limit 2}.
\end{proof}

Our proof of \eqref{eq:limsup LIL} consists of two parts.
In the first part, Lemma \ref{Lem:limsup LIL along lattice},
we obtain \eqref{eq:limsup LIL} with the limit $t \to \infty$ taken
along the points of a lattice $\delta n$, $n \in \N$ where $\delta>0$ is fixed but arbitrary.
In the second part, we extend the convergence in \eqref{eq:limsup LIL} along lattice sequences
to arbitrary sequences $t \to \infty$.

\noindent
\begin{lemma}	\label{Lem:limsup LIL along lattice}
For every $\delta > 0$, we have
\begin{align}
\limsup_{n \to \infty} \frac{e^{\alpha n\delta/2}}{\sqrt{\log (n\delta)}}(W-W_{n\delta})
&=\Big(\frac{2\sigma^2}{-\alpha m'(\alpha)} W\Big)^{\!1/2}	\quad	\text{{\rm a.\,s.}}	\label{eq:limsup LIL along lattice}
\end{align}
\end{lemma}
\begin{proof}
Fix an arbitrary $\delta>0$ and $r\in \delta \N\cup\{\infty\}$. We claim
that \eqref{eq:Berry-Esseen-type condition} holds for the random variables
\begin{equation*}
T_n \defeq V_{\delta n, r}/ \sqrt{\Var_{\delta n}[V_{\delta n,r}]},	\quad n \in \N_0
\end{equation*}
and $\mathcal{R}_n=\H_{\delta n}$. Conditionally given
$\H_{\delta n}$, $V_{\delta n,r}$ is a weighted sum of
independent, centered random variables to which Lemma~\ref{Lem:Berry-Esseen}
applies.
In particular, \eqref{eq:Berry-Esseen} yields
\begin{align*}
\sup_{y\in\R}& \bigg|\Prob
\bigg(\frac{V_{\delta n,r}}{\sqrt{\Var_{\delta n} [V_{\delta n,r}]}} \leq y \,\bigg|\, \H_{\delta n}\bigg) - \Phi(y)\bigg|    \\
&\leq
C
{\frac{\sum_{u\in \I_{\delta n}} \E_{\delta n}[|W_{\delta n,r}(u) - \E_{\delta n}[W_{\delta n,r}(u)] |^3]}{
 (\Var_{\delta n}[V_{\delta n,r}])^{3/2}}}    \\
&\leq
8C\frac{\sum_{u\in\I_{\delta n}} \E_{\delta n}[|W_{\delta n,r}(u)|^3]}{(\Var_{\delta n}[V_{\delta n,r}])^{3/2}},
\end{align*}
where $C>0$ is a finite absolute constant.
In view of \eqref{eq:asymptotic var V_t,r},  the condition \eqref{eq:Berry-Esseen-type condition} will follow from
the almost sure finiteness of
\begin{align}    \label{eq:I defined}
I &\defeq
\sum_{n\geq 0} e^{3\alpha\delta n/2} \sum_{u\in\I_{\delta n}} \E_{\delta n}[|W_{\delta n,r}(u)|^3].
\end{align}
Here, $\E_{\delta n}[|W_{\delta n,r}(u)|^3] = e^{-3\alpha S(u)}\int_{[0,\infty)} x^3 \1_{\{e^{-\alpha\delta n/2} e^{\alpha S(u)}\geq x\}} \, {\rm d} F_{r+\delta n-S(u)}(x)$.
If $S(u) > \delta n + r$, then $F_{r+\delta n-S(u)}$ is the distribution function of the Dirac measure at $0$ and hence the integral vanishes.
Otherwise,  since $F_t \geq G_t \geq G_{\infty}$ pointwise for any $t>0$, we may estimate
\begin{align*}
\int_{[0,\infty)} &e^{3\alpha\delta n/2-3\alpha S(u)}x^3 \1_{\{e^{-\alpha\delta n/2} e^{\alpha S(u)}\geq x\}} \, {\rm d} F_{r+\delta n-S(u)}(x)\\
&\leq \int_{[0,\infty)}\min\{e^{3\alpha\delta n/2-3\alpha S(u)}x^3,1\}\, {\rm d} F_{r+\delta n-S(u)}(x) \\
&\leq \int_{[0,\infty)}\min\{e^{3\alpha\delta n/2-3\alpha S(u)}x^3,1\}\, {\rm d} G_{\infty}(x).
\end{align*}
As in Section \ref{sect:Nerman}, denote by $(\mathsf{S}_n)_{n \in \N_0}$ a random walk
with i.\,i.\,d.\ increments $\mathsf{S}_n-\mathsf{S}_{n-1}$, $n \in \N$
and increment law $\Prob(\mathsf{S}_{n}-\mathsf{S}_{n-1} \in \dx) \defeq e^{-\alpha x} \mu(\dx)$. By the many-to-one lemma \cite[Section 1.3]{Shi:2015},
\begin{align*}
\E[I]
& \leq \int_{[0,\infty)} \E\bigg[\sum_{n \geq 0}
\sum_{u \in \I_{\delta n}} \min\{e^{3\alpha\delta n/2-3\alpha S(u)}x^3,1\}\bigg]\, {\rm d} G_{\infty}(x)	\\
&= \int_{[0,\infty)}\E\bigg[\sum_{n\geq 0}
 \min\{e^{3\alpha\delta n/2-3\alpha  \mathsf S_{\tau_{\delta n}}}x^3,1\}e^{\alpha \mathsf S_{\tau_{\delta n}}}\bigg]\, {\rm d} G_{\infty}(x),
\end{align*}
where $\tau_{\delta n}=\inf\{k \in \N_0: \mathsf S_k> \delta n\}$ is the first-passage time
of the level $\delta n$ of the random walk $(\mathsf{S}_n)_{n \in \N_0}$ for $n \in \N_0$.
The integrand above can be estimated as follows
\begin{align*}
\sum_{n\geq 0}  &
\min\{e^{3\alpha\delta n/2-3\alpha \mathsf{S}_{\tau_{\delta n}}}x^3,1\}e^{\alpha\mathsf {S}_{\tau_{\delta n}}}
\le \sum_{n\geq 0}
\min\{e^{-3\alpha \mathsf{S}_{\tau_{\delta n}}/2 }x^3,1\}e^{\alpha\mathsf  S_{\tau_{\delta n}}}\\
&=x^3\sum_{n:\mathsf S_{\tau_{\delta n}}\ge\frac{2}{\alpha}\log x} e^{-\alpha\mathsf  S_{\tau_{\delta n}}/2 }
+\sum_{n:\mathsf S_{\tau_{\delta n}}
\leq \frac{2}{\alpha}\log x}e^{\alpha\mathsf  S_{\tau_{\delta n}} }.
\end{align*}
Let $\mathsf{X}$ be a copy of $\mathsf S_1$ independent of $(\mathsf S_n)_{n \in \N_0}$.
Since for any $i \in \N$ the set $\{n \in \N_0:\tau_{\delta n}= i\}$ has cardinality at most $\delta^{-1}(\mathsf S_{i}-\mathsf S_{i-1})+1$,
for any real $y$ and $\beta>0$, we may estimate
\begin{align*}
\E\bigg[\sum_{n:\mathsf S_{\tau_{\delta n}}\ge y}e^{-\beta \mathsf S_{\tau_{\delta n}}}\bigg]
&\le \E\bigg[\sum_{i\ge 1}\1_{\{\mathsf S_i\ge y\}}e^{-\beta \mathsf S_i}(\delta^{-1}(\mathsf S_{i}-\mathsf S_{i-1})+1)\bigg]	\\
&= e^{-\beta y} \E\bigg[\sum_{i \geq 0}\1_{\{\mathsf S_i+\mathsf{X}- y \geq 0\}}
e^{-\beta (\mathsf S_{i}+ \mathsf{X}-y)}(\delta^{-1} \mathsf{X}+1)\bigg].
\end{align*}
Now let $f(x) \defeq e^{\beta x} \1_{(-\infty,0]}(x)$ for $x \in \R$,
and let $\mathsf{U}$ denote the renewal measure associated with $(\mathsf S_n)_{n \in \N_0}$,
i.e., $\mathsf{U}(B) = \sum_{n \in \N_0} \Prob(S_n \in B)$ for Borel sets $B \subseteq \R$.
The function $f$ is directly Riemann integrable, hence $f*\mathsf{U}(x) \defeq \int f(x-u) \, \mathsf{U}(\du)$
is bounded by some finite constant $C>0$ (that may depend on $\beta$)
by the key renewal theorem. Consequently,
\begin{align*}
e^{-\beta y} & \E\bigg[\sum_{i \geq 0}\1_{\{\mathsf S_i+\mathsf{X}- y \geq 0\}}
e^{-\beta (\mathsf S_{i}+ \mathsf{X}-y)}(\delta^{-1} \mathsf{X}+1)\bigg]	\\
&= e^{-\beta y} \E\big[f*\mathsf{U}(y-x) (\delta^{-1} \mathsf{X}+1)\big]		\\
&\leq Ce^{-\beta y}\E[\delta^{-1} \mathsf{X}+1]
= C(-\delta^{-1}m'(\alpha)+1) e^{-\beta y}.
\end{align*}
Similarly, by increasing the value of $C>0$ if necessary,
\begin{align*}
\E\bigg[\sum_{n:\mathsf S_{\tau_{\delta n}} \leq y}e^{\beta \mathsf S_{\tau_{\delta n}}}\bigg]
&\leq
\E\bigg[\sum_{i \geq 1}\1_{\{\mathsf{S}_i \leq y\}} e^{\beta \mathsf{S}_i}(\delta^{-1}(\mathsf{S}_{i}-\mathsf{S}_{i-1})+1)\bigg]	\\
&=e^{\beta y}\E\bigg[\sum_{i \geq 0} \1_{\{\mathsf S_i+\mathsf{X}- y\le 0\}} e^{\beta (\mathsf{S}_i+\mathsf{X}-y)}(\delta^{-1} \mathsf{X}+1)\bigg]\\
&\leq Ce^{\beta y}\E[\delta^{-1} \mathsf{X}+1] = C(-\delta^{-1}m'(\alpha)+1) e^{\beta y}.
\end{align*}
Hence, for some constant $C'>0$,
\begin{align*}
\E\bigg[\sum_{n\geq 0}
\min\{e^{3\alpha\delta n/2-3\alpha S_{\tau_{\delta n}}}x^3,1\}e^{\alpha \mathsf{S}_{\tau_{\delta n}}}\bigg] \leq Cx^2,
\end{align*}
which, in turn gives,
\begin{align*}
\E[I] \leq C \int_{[0,\infty)} x^2\, {\rm d} G_{\infty}(x) = C \E[\sup_{s \geq 0} (W_s\!-\!1)^2] < \infty
\end{align*}
in view of \eqref{eq:Doob's max L^2 W_t}.
In particular, $I<\infty$ almost surely and the condition \eqref{eq:Berry-Esseen-type condition} is fulfilled.
\smallskip

\noindent
An appeal to Lemma \ref{Lem:Asmussen+Hering}
with $T_n = V_{\delta n,r}/ \sqrt{\Var_{\delta n} [V_{\delta n,r}]}$
in combination with \eqref{eq:asymptotic var V_t,r} gives, for fixed $r\in\delta\N$,
\begin{equation}\label{eq:limsup V_delta n,r}
\limsup_{n\to\infty}\sqrt{\frac{e^{\alpha\delta n}}{\log (\delta n)}} V_{\delta n,r} = \sqrt{2 c_r W}	\quad\text{a.\,s.}
\end{equation}
because $V_{\delta n,r}$ is $\H_{\delta n+r}$-measurable; whereas
\begin{equation}	\label{eq:limsup V_delta n,infty}
\limsup_{n\to\infty}\sqrt{\frac{e^{\alpha\delta n}}{\log (\delta n)}} V_{\delta n,\infty}
\leq	\sqrt{2 c_\infty W} = \Big(\frac{2\sigma^2}{-\alpha m'(\alpha)} W \Big)^{\! 1/2}	\quad	\text{a.\,s.}
\end{equation}
Next, we shall prove that \eqref{eq:limsup V_delta n,r} and \eqref{eq:limsup V_delta n,infty} entail
\begin{equation}	\label{eq:limsup W_delta n+r-W_delta n}
\limsup_{n\to\infty} \sqrt{\frac{e^{\alpha\delta n}}{\log (\delta n)}}(W_{\delta n +r}-W_{\delta n})
= \sqrt{2 c_r W}	\quad	\text{a.\,s.}
\end{equation}
for fixed $r\in\delta\N$ and
\begin{equation}	\label{eq:limsup W-W_delta n}
\limsup_{n\to\infty}\sqrt{\frac{e^{\alpha\delta n}}{\log
(\delta n)}}(W-W_{\delta n}) \leq \sqrt{2 c_\infty W}
= \Big(\frac{2\sigma^2}{-\alpha m'(\alpha)} W\Big)^{\!1/2}
\quad\text{a.\,s.}
\end{equation}
To this end, it is
enough to check that, for $r\in \delta\N\cup\{\infty\}$,
\begin{align}
\lim_{n\to\infty} e^{\alpha\delta n/2} \sum_{u\in\I_{\delta n}}e^{-\alpha S(u)} & |[W_{r+\delta n - \cdot}]_u \circ S(u)-1|	\notag	\\
&\cdot \1_{\{e^{\alpha\delta n/2}e^{-\alpha S(u)}|[W_{r+\delta n - \cdot}]_u \circ S(u)-1|> 1\}}=0	\quad	\text{a.\,s.}	\label{eq:rel1}
\end{align}
and
\begin{equation}	\label{eq:rel2}
\lim_{n\to\infty} e^{\alpha\delta n/2} \sum_{u\in\I_{\delta n}} |\E_{\delta n}[W_{\delta n,r}(u)]|
= 0	\quad	\text{a.\,s.}
\end{equation}
Since, for $u \in \I_{\delta n}$, $\E_{\delta n}[[W_r]_u-1]=0$ and $S(u)$ is $\H_{\delta n}$-measurable, we have
\begin{align*}
\big|\E_{\delta n}&[W_{\delta n,r}(u)]\big|	\\
&=
\big|\E_{\delta n} \big[e^{-\alpha S(u)}([W_{r+\delta n - \cdot}]_u \circ S(u)\!-\!1)
\1_{\{e^{\alpha\delta n/2}e^{-\alpha S(u)}|[W_{r+\delta n - \cdot}]_u \circ S(u)-1|\leq 1\}}\big]\big|	\\
&= \big|\E_{\delta n} \big[e^{-\alpha S(u)}([W_{r+\delta n - \cdot}]_u \circ S(u)\!-\!1)\1_{\{e^{\alpha\delta n/2} e^{-\alpha S(u)}|[W_{r+\delta n - \cdot}]_u \circ S(u)-1|> 1\}}\big]\big|
\\
&\leq \E_{\delta n} \big[e^{-\alpha S(u)}|[W_{r+\delta n - \cdot}]_u \circ S(u)\!-\!1|\1_{\{e^{\alpha\delta n/2}e^{-\alpha S(u)} |[W_{r+\delta n - \cdot}]_u \circ S(u)-1|> 1\}}\big].
\end{align*}
Hence, both relations~\eqref{eq:rel1} and~\eqref{eq:rel2} follow if
we can show that
\begin{equation}	\label{eq:series test}
\E\bigg[\sum_{n\geq 0}e^{\alpha\delta n/2}\sum_{u\in\I_{\delta n}}e^{-\alpha S(u)}
\int_{(e^{-\alpha\delta n/2}e^{\alpha S(u)},\, \infty)} \!\!\!\!\! x \, {\rm d}F_{r+\delta n - S(u)}(x) \bigg] < \infty.
\end{equation}
To see this, notice that
\begin{align*}
\E\bigg[\sum_{n\geq 0}e^{\alpha\delta n/2}\sum_{u\in\I_{\delta n}} & e^{-\alpha S(u)}
\int_{(e^{-\alpha\delta n/2}e^{\alpha S(u)}, \infty)} \!\!\!\!\! x \, {\rm d}F_{r+\delta n - S(u)}(x) \bigg]	\\
&\leq \E\bigg[\sum_{n\geq 0}e^{\alpha\delta n/2}\sum_{u\in\I_{\delta n}} e^{-\alpha S(u)}
\int_{(e^{\alpha(\delta n/2 + r)}, \infty)} \!\!\!\!\! x \, {\rm d}G_r(x) \bigg]	\\
&= \sum_{n\geq 0}e^{\alpha\delta n/2}\int_{(e^{\alpha(\delta n/2 + r)}, \infty)} \!\!\!\!\! x \, {\rm d}G_r(x) 	\\
&\leq \int_{(1,\infty)}x\bigg(\sum_{n=0}^{(2/\delta)(\frac{\log x}\alpha - r)}e^{\alpha\delta n/2} \bigg) \, {\rm d}G_r(x)	\\
&\leq \mathrm{const} \cdot \int_{(1,\infty)} x^2 \, {\rm d}G_r(x)< \infty.
\end{align*}
The proof of \eqref{eq:limsup W_delta n+r-W_delta n} and \eqref{eq:limsup W-W_delta n} is complete.

It remains to show that ``$\leq$'' can be replaced by ``$=$''
in~\eqref{eq:limsup W-W_delta n}. As has already been remarked at the beginning of the
proof, once we have proved \eqref{eq:limsup W-W_delta n}, we also have
\begin{equation}	\label{eq:liminf W-W_delta n}
{\lim\inf}_{n\to\infty} \sqrt{\frac{e^{\alpha\delta n}}{\log
(\delta n)}}(W-W_{\delta n})\geq -\Big(\frac{2 \sigma^2}{-\alpha m'(\alpha)} W\Big)^{\! 1/2}	\quad	\text{a.\,s.}
\end{equation}
For any $r\in \delta \N$, the following equality holds
\begin{multline*}
\sqrt{\frac{e^{\alpha \delta n}}{\log (\delta n)}}(W-W_{\delta n})
=\sqrt{\frac{e^{\alpha(\delta n+r)}}{\log (\delta n+r)}}(W-W_{\delta n+r})
\sqrt{\frac{\log (\delta n+r)}{\log (\delta n)}} e^{-\alpha \delta r/2}\\+\sqrt{\frac{e^{\alpha\delta n}}{\log (\delta n)}}(W_{\delta n+r}-W_{\delta n}).
\end{multline*}
From \eqref{eq:limsup W_delta n+r-W_delta n} and \eqref{eq:liminf W-W_delta n} we infer
\begin{align*}
\limsup_{n\to\infty} & \sqrt{\frac{e^{\alpha\delta n}}{\log (\delta n)}}(W-W_{\delta n})	\\
&\geq \liminf_{n\to\infty} \sqrt{\frac{e^{\alpha(\delta n+r)}}{\log (\delta n+r)}}(W-W_{\delta n+r})
\sqrt{\frac{\log(\delta n+r)}{\log (\delta n)}}e^{-\alpha \delta r/2}	\\
&\hphantom{\geq}~+\limsup_{n\to\infty}\sqrt{\frac{e^{\alpha\delta n}}{\log (\delta n)}}(W_{\delta n+r}-W_{\delta n})	\\
&\geq -\Big(\frac{2 \sigma^2}{-\alpha m'(\alpha)} W\Big)^{\! 1/2} e^{-\alpha \delta r/2} + \big(2 c_r W\big)^{\! 1/2}.
\end{align*}
Letting $r\to\infty$, we arrive at
\begin{equation*}
\limsup_{n\to\infty} \sqrt{\frac{e^{\alpha\delta n}}{\log (\delta n)}}(W-W_{\delta n})
\geq \big(2 c_\infty W\big)^{\! 1/2} = \Big(\frac{2 \sigma^2}{-\alpha m'(\alpha)} \Big)^{\! 1/2}	\quad \text{a.\,s.}
\end{equation*}
\end{proof}

%The proof of Theorem \ref{Thm:LIL} is very similar to the derivation on p.~130 in \cite{Asmussen+Hering:1983}.

\begin{proof}[Proof of Theorem \ref{Thm:LIL}]
We have to show that \eqref{eq:limsup LIL} with $\delta n$ replacing $t$ entails \eqref{eq:limsup LIL}.
Plainly,
\begin{align*}
\limsup_{t\to\infty} \sqrt{\frac{e^{\alpha t}}{\log t}}(W\!-\!W_t)
\geq \limsup_{n\to\infty} \sqrt{\frac{e^{\alpha\delta n}}{\log (\delta n)}}(W\!-\!W_{\delta n})
= \Big(\frac{2 \sigma^2}{-\alpha m'(\alpha)} W\Big)^{\!1/2}	\quad \text{a.\,s.}
\end{align*}
Thus, it remains to prove the converse inequality
\begin{equation}	\label{eq:limsup LIL W-W_t leq}
\limsup_{t \to \infty} \sqrt{\frac{e^{\alpha t}}{\log t}}(W-W_t) \leq \Big(\frac{2 \sigma^2}{-\alpha m'(\alpha)} W\Big)^{\!1/2}
\quad	\text{a.\,s.},
\end{equation}
To this end,
fix arbitrary $\delta,\rho >0$, let $n \in \N_0$ and notice that
the following particular case of \eqref{eq:limsup W_delta n+r-W_delta n} holds
\begin{equation*}
\limsup_{n\to\infty} \sqrt{\frac{e^{\alpha\delta n}}{\log (\delta n)}}(W_{\delta n}-W_{\delta(n+1)})
= \sqrt{2 c_\delta W}	\quad	\text{a.\,s.}
\end{equation*}
This in combination with the fact that
$W_{\delta n}-W_{\delta(n+1)}$ is $\H_{\delta(n+1)}$-measurable
and the conditional Borel-Cantelli lemma (see, for instance, Theorem 5.3.2 on p.\;240 in \cite{Durrett:2010}) implies
\begin{equation}	\label{eq:conditional BC}
\sum_{n\geq 1}\Prob(W_{\delta n}-W_{\delta(n+1)}>\varepsilon_n|\H_{\delta n}) < \infty
\quad	\text{a.\,s.}
\end{equation}
for
\begin{equation}	\label{eq:epsilon_n}
\varepsilon_n = (1+\rho)\sqrt{\frac{\log (\delta n)}{e^{\alpha \delta n}}} \sqrt{2 c_\delta W_{\delta n}}.
\end{equation}
For $t\in [\delta n, \delta (n+1))$, define
\begin{equation*}
A_t \defeq \E_t[(W_{\delta(n+1)}-W_t)^2]
\end{equation*}
and observe that, in view of Lemma \ref{Lem:asymptotics of W_t(2alpha)}
and the fact that the martingale $(W_t)_{t\geq 0}$ is $L^2$-bounded,
\begin{equation}	\label{eq:aux17}
\sup_{t \geq 0} e^{\alpha t}A_t = \sup_{t \geq 0} e^{\alpha t}\sum_{u\in\I_t}e^{-2\alpha S(u)} v_{\delta(n+1)-S(u)} <\infty
\quad \text{a.\,s.}
\end{equation}
Now let
\begin{equation*}
B_n \defeq \sup_{t \in [\delta n, \delta(n+1))}(W_{\delta n}-W_t - (2A_t)^{1/2})
\end{equation*}
and
\begin{equation*}
t^*_n\defeq\inf\{s\geq \delta n:
W_{\delta n}-W_s-(2A_s)^{1/2} > \varepsilon_n\}.
\end{equation*}
Then $t^*_n$, being the hitting time of an open set by an adapted, right-continuous process,
is an optional time for $(\H_t)_{t \geq 0}$. Thus,
\begin{align*}
\Prob(&W_{\delta n}-W_{\delta(n+1)}>\varepsilon_n|\H_{\delta n})
\geq \Prob(W_{\delta n}-W_{\delta (n+1)}>\varepsilon_n, t^*_n < \delta(n+1)|\H_{\delta n})	\\
&\geq \Prob(W_{t^*_n}-W_{\delta (n+1)}>-(2A_{t^*_n})^{1/2} , t^*_n < \delta(n+1)|\H_{\delta n})	\\
&=\E\Big[\Prob\big(W_{t^*_n}-W_{\delta(n+1)}>-(2A_{t^*_n})^{1/2}\big|\H_{t^*_n}\big)\1_{\{ t^*_n < \delta(n+1)\}} \Big| \H_{\delta n}\Big]	\\
&\geq 2^{-1}\Prob(t^*_n<\delta (n+1)|\H_{\delta n})
=2^{-1}\Prob(B_n>\varepsilon_n|\H_{\delta n})
\end{align*}
having used the definition of $t^\ast_n$ and the fact that $t_n^*$ is optional for the second
inequality, the tower property of the conditional expectations for
the first equality and the Markov inequality for the last inequality.

This entails
\begin{equation*}
\sum_{n\geq 1} \Prob(B_n>\varepsilon_n|\H_{\delta n}) < \infty
\quad	\text{a.\,s.}
\end{equation*}
and thereupon $B_n\leq \varepsilon_n$ eventually a.\,s. Thus,
\begin{align*}
&\limsup_{t\to\infty}\sqrt{\frac{e^{\alpha t}}{\log t}}(W-W_t)	\\
&\leq e^{\alpha\delta/2} \bigg(\limsup_{n\to\infty}\sqrt{\frac{e^{\alpha \delta n}}{\log (\delta n)}}(W\!-\!W_{\delta n})
+\limsup_{n\to\infty}\sqrt{\frac{e^{\alpha \delta n}}{\log (\delta n)}} \sup_{t \in [\delta n,\delta(n+1))} \!\!\! (W_{\delta n}\!-\!W_t)\bigg)	\\
&=e^{\alpha\delta/2} \bigg(\Big(\frac{2\sigma^2}{-\alpha m'(\alpha)} W\Big)^{\!1/2} +\limsup_{n\to\infty}\sqrt{\frac{e^{\alpha \delta n}}{\log (\delta n)}}B_n\bigg)	\\
&\leq e^{\alpha\delta/2} \bigg(\Big(\frac{2\sigma^2}{-\alpha m'(\alpha)} W\Big)^{\!1/2} + (1+\rho) \sqrt{2 c_\delta W}\bigg),
\end{align*}
where the equality is a consequence of \eqref{eq:aux17}.
Recall the definition of $c_\delta$ from \eqref{eq:asymptotics of W_t(2alpha) generalized recalled}
and notice that
$\lim_{\delta\to 0+} c_\delta=0$ by Corollary \ref{Cor:asymptotics of W_t(2alpha)}.
With this at hand, letting $\delta \downarrow 0$ gives \eqref{eq:limsup LIL W-W_t leq}.
\end{proof}

\subsubsection*{Acknowledgements.} This work was initiated during the visit of A. Iksanov to Innsbruck in August 2019. He gratefully acknowledges hospitality and the financial support from the grant ME3625/3-1. A. Iksanov was supported by Ulam programme funded by the Polish  national agency for academic exchange (NAWA),  project  no.PPN/ULM/2019/1/00010/DEC/1.

%
%\begin{appendix}
%
%\section{Auxiliary results}
%
%\begin{proposition}	\label{Prop:filtrations}
%For all $t \geq 0$, we have $\A_{N_t} = \H_t = \F_{\I_t}$.
%\end{proposition}
%
%\end{appendix}

%\bibliographystyle{plain}
%\bibliography{BRW}

\end{document}